\documentclass[11pt]{article}

\usepackage{mathtools}
\usepackage{amsmath}
\usepackage{amsthm}
\usepackage{amssymb}
\usepackage{mathrsfs}
\usepackage{graphicx}
\def\div{{\,\rm div \,}}

\def\g{{\,\rm \gamma \,}}

\def\T{{\cal T}}
\def\Id{{\,\rm Id \,}}

\def\CC{{\,\rm C\,}}
\def\WW{{\,\rm W\,}}
\def\qiq{{\quad\mbox{in}\quad}}
\def\qfq{{\quad\mbox{for}\quad}}
\def\qaq{{\quad\mbox{at}\quad}}
\def\qoq{{\quad\mbox{on}\quad}}

\def\id{{\,\rm id \,}}

\def\sym{{\,\rm sym \,}}

\def\ii{{\,\rm i \,}}

\def\supp{{\,\rm supp \,}}

\def\B{{\,\cal B \,}}

\def\+M{{\,\rm M^{n\times n}_+ \,}}
\def\tr{{\,\rm tr \,}}

\def\qflq{{\quad\mbox{for}\quad}}

\def\qflq{{\quad\mbox{for all}\quad}}

\def\LL{{\,\rm L \,}}

\def\ii{{\,\rm i \,}}

\def\supp{{\,\rm supp \,}}

\def\C{{\cal C }}

\def\lam{\lambda}

\def\ol{\overline}

\def\E{{\cal E}}

\def\X{{\cal X}}

\def\Y{{\cal Y}}
\def\E{{\cal E}}
\def\L{{\cal L}}

\def\A{{\cal A}}

\def\P{{\cal P}}

\def\N{{\cal N}}

\def\ka{{\kappa}}

\newfont{\Blackboard}{msbm10 scaled 1200}

\newfont{\roma}{cmr10 scaled 1200}

\def\D{{\cal D}}
\def\II{{\cal I}}

\def\<{{\langle}}
\def\>{{\rangle}}

\def\ga{\gamma}
\def\Ga{\Gamma}
\def\var{\varphi}

\def\a{\alpha}
\def\b{\beta}

\def\Om{\Omega}

\newtheorem{thm}{{}\hskip\parindent Theorem}[section]
\newtheorem{lem}{{}\hskip\parindent Lemma}[section]
\newtheorem{pro}{{}\hskip\parindent
Proposition}[section]

\newtheorem{dfn}{{}\hskip\parindent
Definition}[section]

\def\pl{\partial}
\def\rw{\rightarrow}

\def\be{\begin{equation}}
\def\ee{\end{equation}}
\def\beq{\arraycolsep=1.5pt\begin{eqnarray}}
\def\eeq{\end{eqnarray}}

\def\P{\cal P}
\def\R{I\!\!R}

\def\n{\mathbf{n}}
\title{Rigidity of Strain Tensors for Surfaces With Mixed Type and Applications To Shell Theory }
\date{}
\author{
Liang-Biao Chen and Peng-Fei Yao\\[0.2cm]
\nonumber
School of Mathematics and Statistics\\\nonumber
Key Laboratory of Complex Systems and Data Science of Ministry of
Education\\\nonumber
Shanxi University, Taiyuan, 030006,
China\\\nonumber
e-mail: pfyao@iss.ac.cn}

\begin{document}
\maketitle
\footnote{This work is supported by the National
Science Foundation of China, grant no. 12071463, and by the special fund for Science and Technology Innovation Teams of Shanxi Province, grant no. 202204051002015.}
\begin{quote}
\begin{small}
{\bf Abstract} \,\,\,This paper investigates the rigidity of strain tensors on surfaces with sign-changing Gaussian curvature (mixed-type surfaces) and applies the results to determine the optimal thickness exponent in the first Korn inequality for thin shells. Using tools from Riemannian geometry and generalized tensor analysis, we derive an infinitesimal rigidity lemma for strain tensors, which establishes \(L^2\) regularity estimates for displacements decomposed into tangential and normal components. Specifically, we show that for a mixed-type shell with a middle surface \(S = S^+ \cup \Gamma_0 \cup S^-\) (where \(S^+\), \(S^-\) have positive and negative curvature, respectively, and \(\Gamma_0\) is a parabolic interface), the optimal constant in Korn’s inequality scales as \(h^{4/3}\), matching the behavior previously established for hyperbolic shells. This result is obtained via a combination of geometric decomposition, Fredholm theory for linear operators, and compactness arguments to handle the curvature transition across \(\Gamma_0\). The findings bridge the gap between elliptic and hyperbolic shell theories, providing a unified framework for understanding rigidity in complex geometries with mixed curvature. The derived estimates are shown to be sharp, offering critical insights for the mechanical design of thin-walled structures with non-uniform curvature.
\end{small}
\end{quote}

\section{Introduction and the main results}
\def\theequation{1.\arabic{equation}}
\hskip\parindent The theory of thin shells occupies a pivotal  position in elasticity and geometric analysis, with Korn’s inequalities serving as fundamental tools for establishing the well-posedness of boundary value problems. These inequalities characterize the relationship between deformation gradients and strain tensors, with particular emphasis on how the geometry of the middle surface of a shell  influences its rigidity. A  challenge in this field is to establish the optimal thickness  exponent in Korn’s inequalities for shells with middle surfaces having complex geometries, especially those where the Gaussian curvature changes sign—a scenario that gives rise to ``mixed-type" shells combining elliptic (positive curvature), parabolic (zero curvature), and hyperbolic (negative curvature) regions. Such geometries are ubiquitous in engineering applications, from aerospace structures to biological membranes, yet their analysis remains technically demanding due to the abrupt transitions in curvature-dependent rigidity.

The optimal exponent problem in Korn’s first inequality for thin shells, which relates the \(L^{2}\) norm of displacement gradients to that of symmetric gradients (strains), has been extensively studied for surfaces of uniform curvature: 
\begin{itemize}
\item{For elliptic shells (\(\kappa > 0\)), the optimal thickness exponent is $h$ \cite{Ha2, LeMoPa1, Yao2018} corresponding to the ``pure bending" regime.}
\item{For hyperbolic shells (\(\kappa < 0\)), the exponent \(h^{4/3}\) emerges due to the hyperbolic nature of strain equations \cite{Ha2, Yao2019}.}
\item{For parabolic shells (\(\kappa = 0\)), the scaling \(h^{3/2}\) has been established for cylindrical geometries \cite{GH,GH1,Yao2018}.}
\end{itemize}
However, these results rely heavily on the uniformity of curvature, leaving a critical gap in understanding mixed-type shells, where curvature transitions introduce complex interactions between elliptic and hyperbolic regimes. Such transitions can lead to singular behaviors in strain tensors \cite{CY1}, making it non-trivial to extend classical rigidity estimates. This paper addresses this gap by studying a prototypical mixed-type surface with a parabolic interface \(\Gamma_{0}\) separating positive and negative curvature regions (\(S^{+}\) and \(S^{-}\)).

The primary objectives of this work are twofold:
\begin{itemize}
\item{{\bf Infinitesimal Rigidity Lemma}: We derive an \(L^{2}\) regularity estimate for displacements on mixed-type surfaces, formalized as:
$$\|W\|_{L^{2}(S, T)}^{2} + \|w\|_{[W^{1,2}(S)]'}^{2} \leq C\left(\|U\|_{L^{2}(S, T^{2})}^{2} + \|W\|_{L^{2}(\partial S, T)}^{2}
+\|w\|^2_{\WW^{-1,2}(\Ga_{b_1})}\right),$$
where $W$ and $w$ denote the tangential and normal components of displacement, U is the strain tensor and $\Ga_{b_1}$ is the boundary of $S$ in elliptic part defined as (\ref{boundary}) below. This lemma generalizes earlier results for hyperbolic \cite{Yao2019} and elliptic  surfaces \cite{Ci} by accommodating curvature transitions via generalized tensor analysis.}
\item{{\bf Optimal Thickness Exponent in Korn’s Inequality}: By applying the rigidity lemma to thin shells with middle surface S, we establish that the optimal constant in Korn’s inequality scales as \(h^{4/3}\), matching the hyperbolic regime despite the presence of elliptic regions. This  result highlights the dominant role of hyperbolic curvature in determining global rigidity, even when mixed with elliptic components.}
\end{itemize}

Here the analysis hinges on three key techniques: (1)\,\,{\bf Geometric Decomposition:}  The surface S is partitioned into \(S^{+}\), \(S^{-}\), and \(\Gamma_{0}\), with each region analyzed using curvature-appropriate methods (elliptic theory for \(S^{+}\), hyperbolic theory for \(S^{-}\), and parabolic regularization for \(\Gamma_{0}\)); (2)\,\,{\bf Generalized Tensors:} We employ weak formulations and generalized derivatives to handle low-regularity solutions near curvature transitions, avoiding restrictive smoothness assumptions; (3)\,\,{\bf Fredholm Theory and Compactness Arguments:} Linear operators governing strain equations are shown to have compact resolvents, allowing us to leverage Fredholm alternatives to prove existence and uniqueness of solutions of (\ref{1.33}) and (\ref{3.11n}).

To the best of the authors' knowledge, this paper is the first to establish rigidity results for surfaces with sign-changing curvature (Theorems \ref{t1.1} and \ref{t1.2}). Consequently, it demonstrates that the optimal constant in Korn's inequality scales as \(h^{4/3}\) for certain mixed-type shells (Theorem \ref{t1.3}) as the first time.

We state our main results as follows.

Let $M\subset\R^3$ denote a surface with normal $\mathbf{n}.$ Let $S\subset M$ be defined by
\be S=\{\,\a(t,s)\,|\,(t,s)\in[0,a)\times(-b_0,b_1)\,\},\quad a>0,\quad b_0>0,\quad b>0,\label{Q0}\ee where $\a:$ $[0,a)\times[0,b]\rw M$ is a an embedding map forming a two-parameter family of regular curves parameterized by $t$ and $s,$ such that
\be \Pi(\a_t(t,s),\a_t(t,s))>0,\quad\mbox{for all}\quad (t,s)\in[0,a)\times[-b_0,0],\label{Pi2.801}\ee
and $\a(\cdot,s)$ is a closed curve with  period $a$ for each $s\in[-b_0,b_1],$ where $\Pi=\nabla\n$ denotes the second fundamental form of $M.$ Decompose $S$ as
$$S=S^+\cup\Ga_0\cup S^-,$$ where
$$S^+=\{\,\a(t,s)\,|\,(t,s)\in[0,a)\times(0,b_1)\,\},\quad \Ga_0=\{\,\a(t,0)\,|\,t\in[0,a]\,\},$$
$$ S^-=\{\,\a(t,s)\,|\,(t,s)\in[0,a)\times(-b_0,0)\,\}.$$

 {\bf Curvature assumptions}\,\,\, Let \(\kappa\) denote the Gaussian curvature function on M. Assume that $S$ satisfies the following curvature conditions:
\be\kappa(x)>0\qflq x\in S^+\cup\Ga_{b_1};\label{kappa+}\ee
\be\kappa=0,\quad D\kappa(x)\not=0\qflq x\in\Ga_0;\label{kappa0}\ee
\be\kappa(x)<0\qflq x\in S^-\cup\Ga_{-b_0},\label{kappa-}\ee where $D$ is the covariant differential of the induced metric $g$ and
\beq
\Ga_{b_1}=\{\,\a(t,b_1)\,|\,t\in[0,a)\,\},\quad \Ga_{-b_0}=\{\,\a(t,-b_0)\,|\,t\in[0,a)\,\}.\label{boundary}
\eeq

{\bf Regularity assumption}: Throughout this paper, we assume that the middle surface $S$ is of class $\CC^5.$
This ensures that all geometric quantities (such as the second fundamental form 
$\Pi$ and Gaussian curvature $\kappa$) and the differential operators used in our analysis are well-defined and sufficiently smooth for the proofs in Sections 2 and 3.

We emphasize that the positive-curvature region \(S^+\) and negative-curvature region \(S^-\) can be unrestricted, ensuring that \(S = S^+ \cup \Gamma_0 \cup S^-\) is not merely a narrow strip adjacent to \(\Gamma_0\). 

We mention that the geometry of a standard automobile tire (ignoring tread patterns) is a classic and concrete example with ``mixed-type surface". In particularly, a half-torus precisely possesses the geometric structure (\ref{Q0})-(\ref{kappa-}). Other examples typically require explicit construction (e.g., via gluing, deformation, or tailored parametrizations), but they are physically meaningful and appear in architecture, biology (e.g., cell membranes), and engineering (e.g., complex shell structures). All fall under the broader class of “thin-walled structures with non-uniform curvature” studied in the paper. We hope that our work here can deepen the understanding of the deformation behavior and strength of thin shells with midsurfaces of varying curvature.

Let $y\in \WW^{1,2}(S,\R^3)$ denote a displacement of the middle surface $S.$ We decompose $y$ as
$$y=W+w\mathbf{n},\quad w=\<y,\mathbf{n}\>,$$ where $\<\cdot,\cdot\>$ denotes the dot metric in the Euclidean space $\R^3.$ The (linear) strain tensor of the middle surface associated with the displacement $y$ is defined by $$\Upsilon(y)=\sym DW+w\Pi,$$ where $D$ is the Levi-Civita connection of the induced metric $g$ on $S,$
$$\sym DW=\frac12(DW+DW^T),$$ and
$$\Pi(\a,\b)=\<\nabla_\a\n,\b\>\qflq\a,\,\,\b \in T_xM,\quad x\in S$$ is the second fundamental form of surface $M.$

Let $T^k$ denote the all $k$th-order tensor fields on $S.$ Let $L^2(S,T^k)$ be the space of all $k$th-order tensor fields on $S$ equipped with the norm
$$(P_1,P_2)_{\LL^2(S,T^k)}=\int_S\<P_1,P_2\>dg,$$ where
$$\<P_1,P_2\>=\sum_{i_1,\cdots,i_k=1}^2P_1(e_{i_1},\cdots,e_{i_k})P_2(e_{i_1},\cdots,e_{i_k})\qflq x\in S,$$ and $\{e_1,e_2\}$ is an orthonormal basis of $T_xS.$

The following is one of our main results:  an infinitesimal rigidity lemma for the strain tensor of mixed-type middle surfaces.

\begin{thm}\label{t1.1} Assume that $S$ is of $\CC^5$ and satisfies conditions \eqref{Q0}--\eqref{kappa-}. Then there is a constant $C>0$ such that 
\beq
&&\|W\|_{L^2(S,T)}^2+\|w\|^2_{[\WW^{1,2}(S)]'}\nonumber\\
\le&& C(\|\sym DW+w\Pi\|^2_{L^2(S,T^2)}+\|W\|^2_{\LL^2(\pl S,T)}
+\|w\|^2_{\WW^{-1,2}(\Ga_{b_1})})\label{1.3}\eeq
holds for all $y=W+w\n\in \WW^{1,2}(S,\R^3)$. 
\end{thm}

For the uniformly hyperbolic surface $S,$ estimate $(\ref{1.3})$ is given in \cite{Yao2019}. 
 For a uniformly elliptic surface $S,$ the inequality
$$\|DW\|_{L^2(S,T^2)}^2+\|w\|^2_{\LL^2(S)}\leq C(\|\sym DW+w\Pi\|^2_{L^2(S,T^2)}+\|W\|^2_{\LL^2(\pl S,T)})$$ holds for all $y=W+w\n\in \WW^{1,2}(S,\R^3)$, as shown in \cite{Ci,Yao2018}.

Using Theorem 1.1 and following the proof of \cite[Theorem 1.3]{Yao2019}, we obtain the following result.

\begin{thm}\label{t1.2} Assume that $S$ is of $\CC^5$ and satisfies conditions \eqref{Q0}--\eqref{kappa-}. Then
\be \|w\|^2_{L^2(S)}\leq C(\|Dw\|_{L^2(S)}\|\sym DW+w\Pi\|_{L^2(S)}+\|\sym DW+w\Pi\|^2_{L^2(S)})\label{x1.12}\ee
holds for all $y=W+w\mathbf{n}\in \WW^{1,2}(S,\R^3)$
with $y|_{\Ga_{-b_0}\cup\Ga_{b_1}}=0.$
\end{thm}

{\bf Application to elasticity of thin shells.}\,\,\,Suppose that $S$ is the middle surface of a shell of thickness $h>0,$ given by
$${\Om_h}=\{\,x+t\mathbf{n}\,|\,x\in S,\,-h/2<t<h/2\,\}.$$

Using the estimate of Theorem \ref{t1.2} and repeating the proof of \cite[Theorem 1.4]{Yao2019}, we derive the following result. For simplicity, the detailed proof is omitted.

\begin{thm}\label{t1.3} Suppose that the $\CC^5$ region $S$ satisfies $(\ref{Q0})$-$(\ref{kappa-}).$
There are $C>0,$ $h_0>0,$ independent of $h>0,$ such that
\be\|\nabla y\|^2_{L^2({\Om_h})}\leq\frac{C}{h^{4/3}}\|\sym\nabla y\|^2_{L^2({\Om_h})}\label{1.13}\ee
for all $h\in(0,h_0)$ and $y=W+w\mathbf{n}\in \WW^{1,2}({\Om_h},\R^3)$ with $y|_{\Ga_{-b_0}\cup\Ga_{b_1}}=0.$ Moreover, the exponent of the thickness $h$ in $(\ref{1.13})$ is optimal.
\end{thm}

For the uniform hyperbolic surfaces, estimates \eqref{x1.12} and \eqref{1.13} were initially derived in \cite{Ha2} under the assumption that the middle surface is parameterized by a single principal coordinate system. This restrictive assumption is subsequently removed in \cite{Yao2019}, generalizing the results to more general hyperbolic geometries.

Structure of the rest of the article is: First, we present a subsection 1.1 below to explain the overall strategy of the proofs.
Section 2 introduces generalized tensor spaces and derivative operations, providing the analytical framework for mixed-type surfaces. Section 3 proves the infinitesimal rigidity lemma (Theorem \ref{t1.1}) using operator-theoretic methods. Throughout, we emphasize the interplay between geometric curvature and analytical regularity, showcasing how local curvature properties of middle surfaces dictate global rigidity behavior in thin shells.

\subsection{Overview of the Proof Strategy}

\hskip\parindent Before delving into the technical machinery of generalized tensors and operator theory, we provide a high-level overview of the proof strategy for Theorem 1.1, the infinitesimal rigidity lemma. This roadmap is intended to help the reader understand the purpose of the various components that follow and how they fit together to overcome the central challenge of the mixed-type geometry.

The core difficulty lies in the fact that the character of the partial differential equations governing the strain tensor changes from elliptic (in $S^+$) to hyperbolic (in $S^-$) across the parabolic interface $\Ga_0.$ Classical PDE theory treats these two regimes very differently, and a naive patching of solutions is impossible due to potential singularities at the transition.

Our approach is to avoid working directly with the second-order system for the displacement $y=W+w\n.$ Instead, we follow a strategy pioneered in the study of hyperbolic shells \cite{Yao2017} and adapted in \cite{CY1} for the mixed case. The key idea is to {\bf differentiate the strain equation once more} to obtain a new, first-order system for auxiliary fields that encode the ``rotational" part of the displacement. Specifically, we introduce the fields $v$ and $V$ (defined in (\ref{2.28}) and (\ref{V2.29})) which are, in a generalized sense, related to $\div(QW)$ and a covariant derivative of the normal component $w$ (see (\ref{1.26})).

Theorem \ref{t2.1} shows that the original unknowns $(w,W)$ and the new unknowns $(v, V)$ are linked through two coupled systems of equations: (\ref{1.25}) and (\ref{1.26}). Crucially, the system (\ref{1.26}) for $(v, V)$ has a structure that is amenable to analysis on the hyperbolic side $S^-$.

To exploit this, we focus our attention on the subdomain $S_{-\varepsilon}$, which contains the entire elliptic region $S^+$ and a small neighborhood of the interface $\Ga_0.$ On this domain, we construct a carefully designed linear differential operator $\L_0$ (see (\ref{x3.1})). This operator is built from the principal parts of the equations in (\ref{1.26}) and is chosen so that its adjoint $\L_0^*$ appears naturally when we test the system against suitable functions.

A major achievement of our previous work \cite{CY1} is the establishment of {\bf well-posedness} and {\bf regularity} for the boundary value problems associated with $\L_0$ and $\L_0^*$ (summarized here in Theorem \ref{3t.1} and Theorem \ref{3t.2}). This allows us to use $\L_0^*$ as a powerful tool to probe the solution $(v, V).$ In Section 3, we consider an auxiliary problem (\ref{1.33}) that couples $\L_0^*Z$ with an elliptic equation for a scalar function $z.$ The solvability of this problem is governed by the {\bf Fredholm alternative:} a solution exists if and only if the data is orthogonal to the finite-dimensional kernel of the adjoint problem.

This Fredholm framework is the linchpin of our argument. In Theorem \ref{t3.3}, we use it to derive an a priori estimate for $(v, V)$ in the dual space $W^{2,2}\times W^{2,2}.$ The estimate contains a term involving the projection $\P V$ onto the finite-dimensional kernel $\N.$ Such a term cannot appear in our final rigidity estimate, as it would prevent us from controlling $W$ and $w$ solely by the strain $U.$

To eliminate this unwanted term, we employ a classic {\bf compactness-uniqueness argument}. We assume the desired estimate (\ref{3.26}) is false, which allows us to construct a sequence of counterexamples. By the compactness provided by the Fredholm theory, this sequence converges to a non-trivial limit $(w_0, W_0)$ that satisfies the homogeneous strain equation $\sym DW_0 +w_0\Pi=0$ and specific boundary conditions. The final, and most delicate, step is to prove that this limit must be zero. This is done by showing that the corresponding limit fields $(v_0, V_0)$ vanish identically on both the elliptic and hyperbolic sides by leveraging the unique continuation properties of elliptic and hyperbolic PDEs, respectively. This contradiction proves the estimate (\ref{3.26}).

Finally, in the proof of Theorem \ref{t1.1}, we combine this local estimate on $S_{-\varepsilon}$ with a known rigidity estimate on the purely hyperbolic part $S^c_{-\varepsilon}=S\backslash S_{-\varepsilon}$ (from \cite{Yao2019}) to obtain the global result on the entire mixed-type surface $S.$

In summary, the proof is a multi-stage process: (1) lift the problem to a first-order system for auxiliary fields, (2) use a specially constructed operator $\L_0$ and its well-posedness theory to gain control over these fields on the hyperbolic side, (3) apply the Fredholm alternative to obtain an estimate up to a finite-dimensional kernel, and (4) use a compactness-uniqueness argument, relying on the distinct PDE properties on either side of $\Ga_0,$ to remove the kernel and close the estimate.

\setcounter{equation}{0}
\section{Generalized tensors}
\def\theequation{2.\arabic{equation}}
\hskip\parindent Let $m\geq0$ and $k\geq0$ be  integers. Denote by $\CC_0^m(S,T^k)$  all the $k$th-order tensor fields with  continuous $m$th-order derivatives and compact supports on $\ol{S}.$
Denote by $\D_m(S,T^k)$  the  space $\CC^m_0(S,T^k)$ equipped with the following convergence: a sequence $\{\var_n\}\subset\CC_0^m(S,T^k)$ is said to converge to $\var_0\in\CC_0^m(S,T^k)$ if

(1) there exists a subset
$K\subset\subset S$ such that
$$\supp \var_n\subset K\qflq n\geq1;$$

(2) $\sup_{x\in K}|D_{X_1}\cdots D_{X_j}(\var_n-\var_0)|\rw0$ as $n\rw\infty$ for all $X_1,$ $\cdots,$ $X_j\in \CC^m(S,T)$ where $0\leq j\leq m.$

Let $\D'_m(S,T^k)$ consist of all continuous linear functionals on $\D_m(S,T^k).$  An element in $\D_m'(S,T^k)$ is called {\it a generalized tensor  of  order $k$.}

We define the following generalized derivative operations:
\begin{dfn}
For a given $w\in\D_m'(S),$ the generalized gradient $Dw\in\D_{m+1}'(S,T)$ is defined by
\be Dw(X)=-w(\div X)\qflq X\in \D_{m+1}(S,T).\label{n1.2}\ee

For given $W\in\D_m'(S,T)$ and $y\in\D_m'(S,\R^3),$ the generalized differentials  $DW\in\D_{m+1}'(S,T^2)$ and $\nabla y\in\D_{m+1}'(S,T^2)$ are defined by
\be DW(P)=-W(\div P)\qflq P\in\D_{m+1}(S,T^2),\label{x1.4}\ee and
\be \nabla y(P)=y(-\div P+\<P,\Pi\>\n)\qflq P\in \D_{m+1}(S,T^2),\label{n1.3}\ee respectively.

For given $W\in\D_m'(S,T)$ and $U\in\D_m'(S,T^2),$ the generalized divergences $\div W\in\D'_{m+1}(S)$ and $\div U\in\D_{m+1}'(S,T)$ are defined by
\be \div W(z)=-W(Dz)\qflq z\in\D_{m+1}(S),\label{n1.4}\ee and
\be \div U(Z)=-U(DZ)\qflq Z\in\D_{m+1}(S,T),\label{1.5}\ee respectively.

For given $y\in\D_m'(S,\R^3)$ and $X\in\D_{m+1}(S,T),$ we define $\nabla_Xy\in\D'_{m+1}(S,T)$ by
\be\nabla_Xy(Y)=\nabla y(Y\otimes X)\qflq Y\in\D_{m+1}(S,T).\label{1.6}\ee
\(\nabla_X y \in \mathcal{D}_{m+1}'(S, T)\) is called the generalized derivative of y along direction X.
Similarly, for given $W\in\D_m'(S,T)$ and $X\in\D_{m+1}(S,T),$ $D_XW\in\D_{m+1}'(S,T)$ is defined by
\be D_XW(Y)=DW(Y\otimes X)\qflq Y\in\D_{m+1}(S,T).\label{1.7}\ee
\end{dfn}

In the above definitions, the right hand sides of formulas (\ref{n1.2})-(\ref{1.7}) are linear functionals. For instant,
if $W\in\LL^2(S,T),$  then we interpret $W\in\D_m'(S,T)$ as
$$W(Y)=\int_S\<W,Y\>dg\qflq Y\in\D_m(S,T),$$
and similarly for other terms.

\begin{pro}
Given $y\in\D_m'(S,\R^3)$ and $X\in\D_{m+1}(S,T),$
\be\nabla_Xy(Y)=-y\Big((\div X) Y+\nabla_XY\Big)\qflq Y\in\D_{m+1}(S,T).\label{1.8}\ee
If $y\in\WW^{1,2}(S,\R^3),$ then
\be\nabla_Xy(Y)=\int_S\<\nabla_Xy,Y\>dg\qflq Y\in\D_{m+1}(S,T).\ee
Similarly, if  $W\in \WW^{1,2}(S,T)$ and $X\in\D_{m+1}(S,T),$ then
\be D_XW(Y)=\int_S\<D_XW,Y\>dg\qflq Y\in\D_{m+1}(S,T),\label{n1.10}\ee where $D_XW$ is defined by $(\ref{1.7}).$
\end{pro}

\begin{proof} Let $x\in S$ be given.  Suppose that $\{E_1,E_2\}$ is a frame normal at $x.$ Then
\beq\div(Y\otimes X)&&=\sum_{ij=1}^2D(Y\otimes X)(E_i,E_j,E_j)E_i=\sum_{ij=1}^2E_j(\<Y,E_i\>\<X,E_j\>)E_i\nonumber\\
&&=\sum_{ij=1}^2(\<D_{E_j}Y,E_i\>\<X,E_j\>+\<Y,E_i\>\<D_{E_j}X,E_j\>)E_i\nonumber\\
&&=D_XY+(\div X)Y\qaq x\qflq X,\,Y\in\D_{m+1}(S,T).\label{1.10}\eeq
Thus (\ref{1.8}) follows directly from definition (\ref{1.6}).

Now, let $y\in\WW^{1,2}(S,\R^3).$ Then $y\in\D_m'(S,\R^3)$ can be identified with a linear functional on  $\D_m(S,T)$ via
$$y(Y)=\int_S\<y,Y\>dg\qflq Y\in\D_m(S,T).$$ Noting that
$$\nabla_XY=D_XY-\<Y\otimes X,\Pi\>\n,$$
it follows from \eqref{1.6}, \eqref{n1.3}, and \eqref{1.10} that
\beq \nabla_Xy(Y)&&=y\Big(-\div(Y\otimes X)+\<Y\otimes X,\Pi\>\n\Big)=\int_S\<y,\,-\div(Y\otimes X)+\<Y\otimes X,\Pi\>\n\>dg\nonumber\\
&&=\int_S\<y,-\nabla_XY-(\div X)Y\>dg=-\int_\Ga\<y,Y\>\<X,\nu\>d\Ga+\int_S\<\nabla_Xy,Y\>dg\nonumber\\
&&=\int_S\<\nabla_Xy,Y\>dg\qflq Y\in\D_{m+1}(S,T).\nonumber\eeq A similar argument establishes \eqref{n1.10}.
\end{proof}

Let $\E_m(S,T^k)$ denote the space $\CC^m(S,T^k)$  equipped  with the following convergence: A sequence $\{\var_n\}\subset\CC^m(S,T^k)$ is said to converge to  $\var_0\in\CC^m(S,T^k)$ if for any  compact set
$K\subset\subset S$ and any  $X_1,$ $\cdots,$ $X_j\in\CC^m(S,T)$ where $0\leq j\leq m,$ there holds
$$\lim_{n\rw0}\sup_{x\in K}|D_{X_1}\cdots D_{X_j}(\var_n-\var_0)|=0.$$
Denote by $\E_m'(S,T^k)$ all the continuous linear functionals on $\E_m(S,T^k).$

We define some multiplier operations between generalized tensors as follows.
\begin{dfn} For given $w\in\D_m'(S)$ and $f\in\E_m(S),$  the product $fw\in\D_m'(S)$ is defined by
\be (fw)(z)=w(fz)\qflq z\in\D_m(S);\label{1.12}\ee
for given $W\in\D_m'(S,T)$ and $f\in\E_m(S),$  the product $fW\in\D_m'(S,T)$ by
\be (fW)(Z)=W(fZ)\qflq Z\in\D_m(S,T);\ee
for given $U\in\D_m'(S,T^2)$ and $f\in\E_m(S),$ the product $fU\in\D_m'(S,T^2)$ by
\be (fU)(P)=U(fP)\qflq P\in\D_m(S,T^2);\ee
for given $W\in\D_m'(S,T)$ and $F\in\E_m(S,T),$ the pairing $(\<W,F\>)\in\D_m'(S)$ by
\be (\<W,F\>)(z)=W(zF)\qflq z\in\D_m(S);\label{n2.15}\ee
for given $w\in\D_m'(S)$ and $X\in\E_m(S,T),$ the product $wX\in\D_m'(S,T)$ by
\be (wX)(Z)=w(\<X,Z\>)\qflq Z\in\D_m(S,T);\label{1.16}\ee
for given $U\in\D_m'(S,T^2)$ and $R\in\E_m(S,T^2),$ $\<U,R\>\in\D_m'(S)$ by
\be (\<U,R\>)(z)=U(zR)\qflq z\in\D_m(S);\label{n1.17}\ee
for given $w\in\D_m'(S)$ and $R\in\E_m(S,T^2),$ the product $wR\in\D_m'(S,T^2)$  by
\be (wR)(P)=w(\<R,P\>)\qflq P\in\D_m(S,T^2);\label{1.17}\ee
for given $U\in\D_m'(S,T^2),$ the trace $\tr U\in\D_m'(S)$  by
\be \tr U(z)=U(z\id)\qflq z\in\D_m(S);\label{1.19}\ee
for given $W\in\D_m'(S,T)$ and $R\in\E_m(S,T^2),$ the product $RW\in\D_m'(S,T)$  by
\be (RW)(Z)=W(R^TZ)\qflq Z\in\D_m(S,T);\label{1.19}\ee
for given $U\in\D_m'(S,T^2)$ and $R_1,$ $R_2\in\E_m(S,T^2),$ the product $R_1UR_2\in\D_m'(S,T^2)$  by
\be (R_1UR_2)(P)=U(R_1^TPR_2^T)\qflq P\in\D_m(S,P).\ee
\end{dfn}

We require a linear operator $Q$ (see \cite{CY, CY1, Yao2017, Yao2019}) defined as follows. For each point \(p \in M\), the Riesz representation theorem guarantees the existence of an isomorphism
\be \label{n2.31}
\<\a,Q\b\>=\det\left(\a,\b,\vec{n}(p)\right)\qflq\a,\,\b\in T_pM.\ee
Let $\{e_1, e_2\}$ be an orthonormal basis of $T_pM$ with
positive orientation, i.e.,
$\det\Big(e_1,e_2,\n(p)\Big)=1.$
Then $Q$ can be explicitly expressed as
\be Q\a=\<\a,e_2\>e_1-\<\a,e_1\>e_2\quad\mbox{for
all}\quad\a\in T_pM.\label{n2.32}\ee
Clearly, $Q$ satisfies
$$ Q^T=-Q,\quad Q^2=-\Id.$$
Operator $Q$ plays a pivotal role in our analysis.

For a given or $U\in\D_m'(S,T),$ we define $QU\in\D_m'(S,T^2)$ by
$$(QU)(X\otimes Y)=U(X\otimes Q^TY),\quad (UQ)(X\otimes Y)=U(QX\otimes Y)\qflq X,\,Y\in\D_m(S,T).$$ Moreover, for given $W\in \D'_m(S,T),$ we define $QW\in\D'_m(S,T)$ by
\be (QW)(F)=-W(QF)\qflq F\in D_m(S,T).\label{n2x.24}\ee

\begin{pro}The following formulas hold:
\be \div(wR)=R^TDw+w\div R\qflq w\in\D_m'(S),\quad R\in\E_m(S,T^2),\label{1.22}\ee
\be \div(RW)=\<R,DW\>+\<\div R,W\>\qflq W\in\D_m'(S,T),\quad R\in\E_m(S,T^2),\label{1.23}\ee
\be UQ+QU^T=(\tr U)Q\qflq U\in\D_m'(S,T^2).\label{1.24}\ee
\end{pro}

\begin{proof} Let $x\in S$ be fixed.  Suppose that $\{E_1, E_2\}$ is a frame normal at $x.$ Then
\beq\<R,DZ\>&&=\sum_{ij}R(E_i,E_j)\<D_{E_j}Z,E_i\>\nonumber\\
&&=\sum_{ij}\{E_j[R(E_i,E_j)\<Z,E_i\>]-DR(E_i,E_j,E_j)\<Z,E_i\>\}\nonumber\\
&&=\sum_{j}\{E_j[RZ(E_j)]-\ii_ZDR(E_j,E_j)\}\nonumber\\
&&=\div(RZ)-\<\div R,Z\>\qaq x.\nonumber\eeq
It follows from (\ref{1.5}), (\ref{1.19}), (\ref{1.16}), and (\ref{1.17}) that
\beq\div(wR)(Z)&&=-(wR)(DZ)=-w(\<R,DZ\>)=-w\Big(\div(RZ)\Big)+w(\<\div R,Z\>)\nonumber\\
&&=Dw(RZ)+(w\div R)(Z)=(R^TDw+w\div R)(Z)\nonumber\eeq for all $Z\in\E_{m+1}(S,T).$

Next, for given $z\in\D_{m+1}(S),$ we have at $x$
\beq R^TDz&&=\sum_{ij}E_i(z)R(E_j,E_i)E_j=\sum_{ij}\{E_i[zR(E_j,E_i)]-zDR(E_j,E_i,E_i)\}E_j\nonumber\\
&&=\sum_{j}\{\<\div(zR),E_j\>-z\tr\ii_{E_j}DR\}E_j\nonumber\\
&&=\div(zR)-z\div R.\label{x2.27}\eeq
By (\ref{n1.4}), (\ref{1.19}), (\ref{x2.27}), (\ref{x1.4}), and (\ref{n1.17}),
\beq\div(RW)(z)&&=-(RW)(Dz)=-W(R^TDz)\nonumber\\
&&=DW(zR)+W(z\div R)\nonumber\\
&&=(\<DW,R\>)(z)+(\<W,\div R\>)(z)\qflq z\in D_{m+1}(S).\nonumber\eeq

Finally, we  prove (\ref{1.24}). Let $\{E_1, E_2\}$ be a positively oriented local frame. Then
$$QE_1=-E_2,\quad QE_2=E_1,\quad Q^TE_1=E_2,\quad Q^TE_2=-E_1.$$
Then
$$(UQ+QU^T)(E_1\otimes E_2)=U(QE_1\otimes E_2)+U( Q^TE_2\otimes E_1)=(\tr U)\<QE_1,E_2\>,$$
$$(UQ+QU^T)(E_2\otimes E_1)=(\tr U)\<QE_2,E_1\>,\quad(UQ+QU^T)(E_i\otimes E_i)=(\tr U)\<QE_i,E_i\> $$ for $1\leq i\leq2.$
\end{proof}

Let $S$ be of class $\CC^{m+1}.$ Given $y\in D'_{m}(S,\mathbb{R}^3),$ we define $w\in\D'_m(S),$ $W\in\D'_m(S,T),$ $v\in\D'_{m+1}(S),$ $V\in\D'_{m+1}(S,T),$ and $U\in\D'_{m+1}(S,T^2)$ by
\be w(z)=y(z\n)\qflq z\in\D_{m}(S),\quad W(Z)=y(Z)\qflq Z\in\D_{m}(S,T),\label{2.27}\ee
\be  v(z)=\frac12y(QDz)\qflq z\in\D_{m+1}(S),\label{2.28}\ee
\be V(Z)=y\Big((\div QZ)\mathbf{n}+\nabla\mathbf{n} QZ\Big)\qflq Z\in\D_{m+1}(S,T),\label{V2.29}\ee and
\be U(P)=\nabla y(\sym P)\qflq P\in\D_{m+1}(S,T^2),\label{U2.31}\ee respectively. In the sense of generalized tensors, we have
\be\sym DW+w\nabla\mathbf{n}=U\qiq\D'_{m+1}(S,T^2).\label{n1.25}\ee 

\begin{thm}\label{t2.1}For a given \(y \in \mathcal{D}_m'(S, \mathbb{R}^3)\), the following hold in the sense of generalized tensors:
\be v=\frac12\div QW\qoq  S,\label{n2x.23}\ee
\be\left\{\begin{array}{l}Dw=\nabla\mathbf{n} W-QV\qoq  S,\\
\div W=-(\tr\Pi)w+\tr U\qoq S,\end{array}\right.\label{1.25}\ee
and
\be\left\{\begin{array}{l}Dv=\nabla\mathbf{n} V+Q\div QUQ\qoq S,\\
\div V=-(\tr\Pi)v-\<Q\nabla\mathbf{n},U\>\qoq S.\end{array}\right.\label{1.26}\ee
\end{thm}

\begin{proof}

By (\ref{n2x.24}) and (\ref{n1.4}), we have
$$ 2v(z)=(W+w\nabla\mathbf{n})(QDz)=-(QW)(Dz)=(\div QW)(z)\qflq z\in D_{m+1}(S).$$

For \(Z \in \mathcal{D}_{m+1}(S, T)\), applying \eqref{n1.2} and \eqref{1.19}, we get
\beq(Dw-\nabla\mathbf{n} W)(Z)&&=-w(\div Z)-W(\nabla^T\n Z)=-y\Big((\div Z)\nabla\mathbf{n}+\nabla\mathbf{n} Z\Big)\nonumber\\
&&=-V(Q^TZ)=-(QV)(Z).\nonumber\eeq

By (\ref{n1.3}), (\ref{x1.4}), and (\ref{1.17}), for any \(P \in \mathcal{D}_{m+1}(S, T^2)\),
$$\nabla y(P)=y(-\div P+\<P,\Pi\>\n)=-W(\div P)+w(\<P,\Pi\>)=(DW+w\Pi)(P),$$
which implies
\be\nabla y=DW+w\Pi\qiq \D_m'(S,T^2).\label{1.27}\ee Since
$\div(z\id)=Dz,$  it follows from \eqref{1.17}, \eqref{1.12}, and \eqref{1.19} that
\beq\div W(z)&&=-W(Dz)=-W\Big(\div(z\id)\Big)=DW(z\id)\nonumber\\
&&=\nabla y(z\id)-(w\Pi)(z\id)=\nabla y\Big(\sym(z\id)\Big)-w\Big((\tr\Pi)z\Big)\nonumber\\
&&=U(z\id)-(\tr\Pi)w(z)=(\tr U)(z)-(\tr \Pi)w(z)\qflq z\in\D_{m+1}(S),\nonumber\eeq proving the second equation in (\ref{1.25}).

To prove \eqref{1.26}, let \(x \in S\) and let \(\{ E_1, E_2 \}\) be a positively oriented orthonormal frame at $x,$ with \(E_1 = QE_2\). By the Ricci identity for $Z,$ $X_1,$ $X_2,$ and $X_3\in\D_{m+1}(S,T),$
$$D^2Z(X_1,X_2,X_3)=D^2Z(X_1,X_3,X_2)+R(X_2,X_3,Z,X_1),$$ 
we compute
\beq \<\div\sym DZ,E_2\>&&=E_1(\sym DZ)(E_2,E_1)+E_2(\sym DZ)(E_2,E_2)\nonumber\\
&&=\frac12E_1[DZ(E_2,E_1)+DZ(E_1,E_2)]+E_2[DZ(E_2,E_2)]\nonumber\\
&&=\frac12[D^2Z(E_2,E_1,E_1)+D^2Z(E_1,E_2,E_1)]+E_2(\div Z)-D^2Z(E_1,E_1,E_2)\nonumber\\
&&=E_2(\div Z)+\frac12[D^2Z(E_2,E_1,E_1)-D^2Z(E_1,E_2,E_1)]+\kappa\<Z,E_2\>\nonumber\\
&&=E_2(\div Z)+\kappa\<Z,E_2\>+\frac12E_1(\<D_{E_1}Z, E_2\>-\<D_{E_2}Z,E_1\>)\nonumber\\
&&=E_2(\div Z)+\kappa\<Z,E_2\>+\frac12\<D\div QZ,E_1\>\nonumber\\
&&=\<D\div Z+\kappa Z-\frac12QD\div QZ,\,\,E_2\>\qflq Z\in\D_{m+2}(S,T).\nonumber\eeq
A similar computation yields
$$\<\div\sym DZ,E_1\>=\<D\div Z+\kappa Z-\frac12QD\div QZ,\,\,E_1\>\qaq x.$$ Thus,
\be\div\sym DZ=D\div Z+\kappa Z-\frac12QD\div QZ\qflq Z\in\D_{m+2}(S,T).\label{1.28}\ee

By (\ref{1.5}) and (\ref{1.28})
\beq(\div \sym DW)(Z)&&=-DW(\sym DZ)=W(\div\sym DZ)\nonumber\\
&&=W(D\div Z+\kappa Z-\frac12QD\div QZ)\nonumber\\
&&=(D\div W+\kappa W)(Z)-\frac12y(QD\div QZ)\nonumber\\
&&=(D\div W+\kappa W)(Z)-v(\div QZ)\nonumber\\
&&=(D\div W+\kappa W-QDv)(Z)\qflq Z\in\D_{m+2}(S,T).\nonumber\eeq
Using the relations
$$\nabla\mathbf{n} Q\nabla\mathbf{n}=\ka Q,\quad \nabla\mathbf{n}-(\tr \Pi) \id=Q\nabla\mathbf{n} Q,\quad \div\nabla\mathbf{n}=D(\tr\Pi),$$ and applying \eqref{1.25}, \eqref{1.22}, (\ref{n1.25}), we derive 
\beq \nabla\mathbf{n} V(Z)&&=V(\nabla\mathbf{n} Z)=QV(Q\nabla\mathbf{n} Z)=(\nabla\mathbf{n} W-Dw)(Q\nabla\mathbf{n} Z)\nonumber\\
&&=W(\nabla\mathbf{n} Q\nabla\mathbf{n} Z)-Dw\Big((\tr\Pi)QZ-\nabla\mathbf{n} QZ\Big)\nonumber\\
&&=[\kappa W+\nabla\mathbf{n} Dw-(\tr\Pi)Dw](QZ)\nonumber\\
&&=[\kappa W+\div(w\nabla\mathbf{n})-w\div\nabla\mathbf{n}-(\tr\Pi)Dw](QZ)\nonumber\\
&&=[\kappa W+\div U-\div\sym DW-D((\tr\Pi)w)](QZ)\nonumber\\
&&=[\div U-D\div W+QDv-D((\tr\Pi)w)](QZ)\nonumber\\
&&=[Dv-Q\div U+QD(\tr U)](Z)\nonumber\\
&&=[Dv-Q\div\Big(U-(\tr U)\id\Big)](Z)\nonumber\\
&&=[Dv-Q\div QUQ](Z)\qflq Z\in\D_m(S,T),\nonumber\eeq where we used the identity \(U - (\operatorname{tr}U)\mathrm{id} = QUQ\), proving the first equation in \eqref{1.26}.

Finally, noting that
$$\<\Pi,Q\nabla\mathbf{n}\>=\<Q,\Pi\>=0,\quad\div QDz=0,\quad \div(zQ\nabla\mathbf{n})=-\nabla\mathbf{n} QDz,$$
$$\div\Big(z(\tr\Pi)Q\Big)=-QD(z\tr\Pi)\qflq z\in\D_{m+1}(S),$$ we compute using \eqref{1.27}, \eqref{n1.3}, (\ref{n1.17}), and \eqref{1.23}:
\beq \div V(z)&&=-V(Dz)=-y\Big((\div QDz)\nabla\mathbf{n}+\nabla\mathbf{n} QDz\Big)=-W(\nabla\mathbf{n} QDz)\nonumber\\
&&=W\Big(\div(zQ\nabla\mathbf{n})\Big)=-DW(zQ\nabla\mathbf{n})=-\nabla y(zQ\nabla\mathbf{n})+w\Pi(zQ\nabla\mathbf{n})\nonumber\\
&&=-\nabla y(z\sym Q\nabla\mathbf{n})+\frac12\nabla y(z\nabla\mathbf{n} Q^T-zQ\nabla\mathbf{n})+w(z\<\Pi,Q\nabla\mathbf{n}\>)\nonumber\\
&&=-U(zQ\nabla\mathbf{n})+\frac12\nabla y(z\nabla\mathbf{n} Q^T-zQ\nabla\mathbf{n})\nonumber\\
&&=-U(zQ\nabla\mathbf{n})-\frac12\nabla y\Big(z(\tr\Pi)Q\Big)\nonumber\\
&&=-(\<U,Q\nabla\mathbf{n}\>)(z)-\frac12y\Big(-\div(z(\tr\Pi)Q)+z(\tr\Pi)\<Q,\Pi\>\n\Big)\nonumber\\
&&=-(\<U,Q\nabla\mathbf{n}\>)(z)-v(z\tr\Pi)=-(\<U,Q\nabla\mathbf{n}\>)(z)-((\tr\Pi)v)(z)\nonumber\eeq for $z\in\D_{m+1}(Z),$ where we used \(Q\nabla\vec{n} = (Q\nabla\vec{n})^T + (\operatorname{tr}\Pi)Q\). This confirms the second equation in \eqref{1.26}.
\end{proof}

\setcounter{equation}{0}
\section{Proof of Theorem \ref{t1.1}}
\def\theequation{3.\arabic{equation}}
\hskip\parindent 
Let $x\in M$ be fixed and $\{E_1, E_2\}$ be a frame field normal at $x$ with positive orientation. Then for any $Z\in C^2(M,T)$, there holds at $x$ 
\beq
\<\div (DZ)^T,E_1\>
=&&D[(DZ)^T](E_1,E_1,E_1)
+D[(DZ)^T](E_1,E_2,E_2)
\nonumber\\
=&&E_1\<D_{E_1}Z,E_1\>
+E_2\<D_{E_1}Z,E_2\>
\nonumber\\
=&&\<E_1,D\div Z\>+\<-D_{E_1}D_{E_2}Z
+D_{E_2}D_{E_1}Z,E_2\>
\nonumber\\
=&&\<E_1,D\div Z+\ka Z\>.
\nonumber
\eeq
Similar formula also holds for $E_2$ at $x\in M$. By the arbitrariness of $x\in M$, we have 
\beq
\div (DZ)^T=D\div Z+\ka Z\quad\mbox{holds on}\quad M.
\label{divsymDZ}
\eeq
 Let ${S}\subset M$ be  a bounded Lipschitz region with boundary $\Ga.$ Consider operator $\B:$ $\LL^2({S},T)\rw\LL^2({S},T)$ defined by
$$\B Z=2\div \sym DZ-D\div Z\quad\mbox{with domain}\quad D(\B)=\WW^{2,2}({S},T)\cap\WW^{1,2}_0({S},T).$$
By (\ref{divsymDZ}) and \cite[Theorem 1.26]{Yao2011}, we have 
$$\B Z=\div DZ+\ka Z
=-{\bf\Delta}Z+2\ka Z,$$ 
where ${\bf\Delta}$ is the Hodge-Laplacian. By \cite{Ta}, $\B$ is a self-adjoint operator with  compact resolvent. We then have the direct sum
$$\LL^2({S},T)=\R(\B)\oplus \N(\B),$$ where $\R(\B)$ and $\N(\B)$  denote the range and null space of \(\mathcal{B}\), respectively. Moreover, $\R(\B)\ \bot\ \N(\B)$ in $\LL^2({S},T)$ and $\dim\N<\infty.$ 

 Let $\X(S)$ be the closure of $C^1(S,T)$ under norm $\|\cdot\|_{\X(S)}$ defined by 
\[
\|W\|^2_{\X(S)}
=\|W\|^2_{\LL^2({S},T)}+\|\sym DW\|^2_{[\WW^{1,2}({S},T^2)]'}
+\|W\|^2_{\LL^2(\Ga,T)}.
\]
Then $\WW^{1,2}(S,T)\subset\X(S)$. 
\begin{lem}\label{l3.1} There exists constant $C>0$ such that
\be\|W\|^2_{\LL^2({S},T)}\leq C(\|\sym DW\|^2_{[\WW^{1,2}({S},T^2)]'}+\|W\|^2_{\LL^2(\Ga,T)})\qflq W\in \X(S).\label{3.29}\ee
\end{lem}

\begin{proof}

{\bf Step 1.}\,\,\,We claim that (\ref{3.29}) holds for all $W\in\R(\B)\cap\X(S)$. 

 Let $W\in\R(\B)\cap C^1(S,T)$. There exists $Z\in D(\B)\cap\R(\B)$ such that 
$$\B Z=W,\quad Z|_\Ga=0$$ with the estimate
$$\|Z\|_{\WW^{2,2}({S},T)}\leq C\|W\|_{\LL^2({S},T)}.$$
By \cite[Lemma 2.3]{CY}, we obtain 
\beq \|W\|^2_{\LL^2(S,T)}
&&=\int_{S}\<W,\div DZ+\div (DZ)^T-D\div Z\>dg
\nonumber\\
&&=
\int_{\Ga}\big(\<D_WZ,\nu\>+\<D_{\nu}Z,W\>
-\<W,\nu\>\div Z\big)d\Ga\nonumber\\
&&\quad
+\int_{S}\big(-\<DW+(DW)^T,DZ\>+\div W\cdot\div Z\big)dg
\nonumber\\
&&\leq
C\|W\|_{\LL^2(\Ga,T)}\|DZ\|_{\LL^2(\Ga,T^2)}
+C\|\sym DW\|_{[\WW^{1,2}({S},T^2)]'}\|DZ\|_{\WW^{1,2}({S},T^2)}
\nonumber\\
&&\leq C(\|\sym DW\|_{[\WW^{1,2}({S},T^2)]'}+\|W\|_{\LL^2(\Ga,T)})\|Z\|_{\WW^{2,2}({S},T)}\nonumber\\
&&\leq C(\|\sym DW\|_{[\WW^{1,2}({S},T^2)]'}+\|W\|_{\LL^2(\Ga,T)})\|W\|_{\LL^2(S,T)},\nonumber\eeq which proves (\ref{3.29}) for all $W\in\R(\B)\cap\X(S)$ by density. 

{\bf Step 2.}\,\,\,Let $\{\Phi_1, \cdots, \Phi_k\}$ be an orthonormal basis of $\N(\B),$ where $k=\dim\N(\B).$ Note that $\N(B)\subset W^{1,2}_0(S,T)\subset \X(S)$. By Step 1, it suffices to show that 
\beq&&
(\sum_{i=1}^k\a_i^2)^{1/2}+\|\sym DW\|_{[\WW^{1,2}({S},T^2)]'}\nonumber\\
&&\leq C\Big(\|\sym D(\sum_{i=1}^k\a_i\Phi_i+W)\|_{[\WW^{1,2}({S},T^2)]'}+\|W\|_{\LL^2(\Ga,T)}\Big)
\label{3.31}\eeq 
holds for all $(\a_1,\cdots,\a_k)\in\R^k$ and $W\in\R(\B)\cap\X(S).$ Assume for contradiction that it fails. There exist $\{\,(\a_{n1},\cdots,\a_{nk})\,\}\subset\R^k$ and $W_n\in\R(\B)\cap\X(S)$ for all $n\in\mathbb{N}^+$ satisfying 
$$(\sum_{i=1}^k\a_{ni}^2)^{1/2}+\|\sym DW_n\|_{[\WW^{1,2}({S},T^2)]'}=1,$$
$$\|\sum_{i=1}^k\a_{ni}\sym D\Phi_i+\sym DW_n\|_{[\WW^{1,2}({S},T^2)]'}+\|W_n\|_{\LL^2(\Ga,T)}\leq\frac1n\qflq n\in\mathbb{N}^+.$$
Passing to a subsequence, let $(\a_{n1},\cdots,\a_{nk})\rw(\a_{1},\cdots,\a_{k})$ as $n\to\infty$. 
 Then
\beq
W_n\rw 0\qiq\LL^2(\Ga,T),\quad \sym DW_n\rw U_0=-\sum_{i=1}^k\a_i\sym D\Phi_i\qiq [\WW^{1,2}({S},T^2)]'.\quad\quad\label{convergence}
\eeq
 By Step 1 and (\ref{convergence}), $\{W_n\}_{n=1}^\infty$ is a Cauchy series in $\X(S)$, in view that for all $n_1,n_2\in\mathbb{N}^+$, 
\beq
&&\|W_{n_1}-W_{n_2}\|^2_{\X(S)}\nonumber\\
=&&\|W_{n_1}-W_{n_2}\|^2_{\LL^2({S},T)}
+\|\sym D(W_{n_1}-W_{n_2})\|^2_{[\WW^{1,2}({S},T^2)]'}
+\|W_{n_1}-W_{n_2}\|^2_{\LL^2(\Ga,T)}
\nonumber\\
\le&&C\big(\|\sym D(W_{n_1}-W_{n_2})\|^2_{[\WW^{1,2}({S},T^2)]'}
+\|W_{n_1}-W_{n_2}\|^2_{\LL^2(\Ga,T)}\big)
\nonumber\\
\le&&C\big(\|\sym DW_{n_1}-U_0\|^2_{[\WW^{1,2}({S},T^2)]'}
+\|\sym DW_{n_2}-U_0\|^2_{[\WW^{1,2}({S},T^2)]'}
\nonumber\\
&&\qquad\qquad
+\|W_{n_1}\|^2_{\LL^2(\Ga,T)}
+\|W_{n_2}\|^2_{\LL^2(\Ga,T)}
\big).
\nonumber
\eeq
Hence, $W_n\rw W_0$ in $\X(S)$ for some $W_0\in \X(S)\subset\LL^2(S,T)$ with 
\[
\sym DW_0=-\sum_{i=1}^k\a_i\sym D\Phi_i\in \WW^{1,2}(S,T^2).\]
Then $W_0\in \WW^{1,2}_0(S,T)$ by first Korn's inequality and (\ref{convergence}). $W_0$ satisfies elliptic equation 
\[
\B W_0
=2\div \sym DW_0-D\div W_0
=-\sum_{i=1}^k\a_i\B\Phi_i=0
\]
on $S$ in weak sense. We conclude that $W_0\in D(\B)=\WW^{2,2}(S,T)\cap\WW^{1,2}_0(S,T)$ by the standard elliptic equation theory, which yields $W_0\in \N(\B)$. 

On the other hand, we also have $W_0\in \R(\B)$, since $\R(\B)$ is a closed subspace of $\LL^2(S,T)$ and $W_j\rw W_0$ in $\LL^2(S,T)$. 
Hence, $W_0=0$ and $\alpha_i = 0$, $i=1,\cdots,k$, which is a contradiction to 
\[(\sum_{i=1}^k\a_{i}^2)^{1/2}+\|\sym DW_0\|_{[\WW^{1,2}({S},T^2)]'}=1.
\]

\end{proof}

\begin{lem}\label{3l.2} The following estimate holds:
\be \|v\|_{[\WW^{1,2}({S})]'}\leq C(\|Dv\|_{[\WW^{2,2}({S},T)]'}+\|v\|_{[\WW^{3/2,2}(\Ga)]'})\qflq v\in\WW^{1,2}({S}).\label{3x.5}\ee
\end{lem}

\begin{proof} For a given $f\in\WW^{1,2}({S}),$ we solve the problem
$$-\Delta z=f\qflq x\in{S},\quad z|_\Ga=0.$$ Then
$\|z\|_{\WW^{3,2}({S})}\leq C\|f\|_{\WW^{1,2}({S})}.$
By integration by parts, we have
\beq(v,f)_{\LL^2({S})}&&=-\int_\Ga v\<Dz,\nu\>d\Ga+(Dv,Dz)_{\LL^2({S},T)}\nonumber\\
&&\leq\|v\|_{[\WW^{3/2,2}(\Ga)]'}\|Dz\|_{\WW^{3/2,2}(\Ga,T)}+\|Dv\|_{[\WW^{2,2}({S},T)]'}\|Dz\|_{\WW^{2,2}(\Ga,T)}\nonumber\\
&& \leq C(\|Dv\|_{[\WW^{2,2}({S},T)]'}+\|v\|_{[\WW^{3/2,2}(\Ga)]'})\|f\|_{\WW^{1,2}({S})}.\nonumber\eeq
Thus, \eqref{3x.5} follows by the definition of the dual norm.
\end{proof}

\begin{lem}\label{3l.3} Let ${S}\subset M$ be  a region whose boundary \(\Gamma\) consists of finitely many closed curves. Then
\be\|\div\div P\|_{[\WW^{2,2}({S})]'}\leq C\|P\|_{\LL^2({S},T^2)}\qflq P\in\LL^2({S},T^2).\label{3n.21}\ee
\end{lem}

\begin{proof}
Assume $P\in C^2_0({S},T^2).$  From \cite[Lemma 2.3]{CY},  the identity
\be \div(PZ)=\<P,DZ\>+\<\div P, Z\>\qflq x\in{S},\quad Z\in\WW^{1,2}({S},T)\label{3x.22}\ee holds. For a given \(z \in W^{2,2}(S)\), applying \eqref{3x.22} and integrating by parts, we obtain
\beq\big|(\div\div P, z)_{\LL^2({S})}\big|
&&=\bigg|\int_{S}[\div(z\div P)-\<\div P,Dz\>]dg\bigg|\nonumber\\
&&=\bigg|\int_\Ga (z\<\div P,\nu\>-\<PDz,\nu\>)d\Ga+\int_{S}\<P,D^2z\>dg\bigg|\nonumber\\
&&\leq C\|P\|_{\LL^2({S},T^2)}\|z\|_{\WW^{2,2}({S})}.\label{3n.22}\eeq
(\ref{3n.21}) follows (\ref{3n.22}) and density.
\end{proof}

Suppose that for a given small $\varepsilon>0,$ a function  $\eta_0\in\CC^\infty_0(-\infty,\infty)$  is given satisfying
$$0\leq\eta_0\leq1;\quad \eta_0=0\qflq s\leq \varepsilon/2;\quad \eta_0=1\qflq s\geq\varepsilon.$$ Set
$$\eta(\a(t,s))=\varsigma\eta_0(s)\qflq (t,s)\in(0,a)\times(-b_0,b_1),$$ where $\varsigma$ is a large number, to be specified later.
For  the given small $\varepsilon>0,$  define
$${S}_{-\varepsilon}=\{\,\a(t,s)\,|\,(t,s)\in(0,a]\times(-\varepsilon,b_1)\,\}.$$
Suppose
$$X\in\CC^m({S}_{-\varepsilon},T)$$ is given in \cite[Section 2, (2.34)]{CY1}.  Let
\be\L_0V
=e^{-\g\kappa}[(\div Q\nabla\mathbf{n} V-\eta\<V,QX\>)QX+(\div V-\eta\<V,\nabla\mathbf{n} X\>)\nabla\mathbf{n} X],\label{defL}\ee
\beq\L_0^*V&&=\nabla\mathbf{n} Q D(e^{-\g\kappa}\<QX,V\>)-D(e^{-\g\kappa}\<\nabla\mathbf{n} X,V\>)\nonumber\\
&&\quad-\eta e^{-\g\kappa}(\<V,QX\>QX+\<V,\nabla\mathbf{n} X\>\nabla\mathbf{n} X),\label{defL*}\eeq for $V\in\LL^2({S}_{-\varepsilon},T).$ Let $m\geq0$ be an integer. By \cite[Lemma 2.3]{CY1} we have
\beq (W,\L_0V)_{\LL^2({S},T)}&&=(V,\L_0^*W)_{\LL^2({S},T)}\nonumber\\
&&\quad+\int_{\pl{S}}(\pounds_1\<V,\nu\>\<W,\nu\>-\pounds_2\<V,QX\>\<W,QX\>)d\Ga\label{3x.7}\eeq
for $V,$ $W\in T{S},$ where
\be\pounds_1=\frac{e^{-\g\kappa}}{\<X,\nu\>}\Pi(X,X),\quad \pounds_2=\frac{e^{-\g\kappa}}{\<X,\nu\>}\Pi(Q\nu,Q\nu) \qflq x\in\pl{S}.\label{n2.17x}\ee

Consider the problems
\be\left\{\begin{array}{l}\L_0V=F,\\
\<V,QX\>|_{\Ga_{-\varepsilon}}=p,\quad\<V,\nu\>|_{\Ga_{b_1}}=q\end{array}\right.\label{3x.9}\ee and
\be\left\{\begin{array}{l}\L_0^*V=F,\\
\<V,\nu\>|_{\Ga_{-\varepsilon}}=p,\quad\<V,QX\>|_{\Ga_{b_1}}=q,\end{array}\right.\label{3x.10}\ee respectively, where $F\in\WW^{m,2}({S}_{-\varepsilon},T),$ $p\in\WW^{m,2}(\Ga_{-\varepsilon}),$ and
$q\in\WW^{m,2}(\Ga_{b_1})$ are given. Here $\Ga_{-\varepsilon}=\{\,\a(t,-\varepsilon)\,|\,t\in[0,a]\,\}.$
It follows from \cite[Theorems 3.1 and 3.2]{CY1} that
\begin{thm}\label{3t.1} Let $S$ be of class $\CC^{m+3}.$ There exists a unique solution $V\in\WW^{m,2}({S}_{-\varepsilon},T)$ to problem $(\ref{3x.9}),$ or $(\ref{3x.10})$ satisfying
\beq &&\|V\|_{\WW^{m,2}({S}_{-\varepsilon},T)}+\|V\|_{\WW^{m,2}(\pl{S}_{-\varepsilon}, T)}\nonumber\\
&&\leq C( \|F\|_{\WW^{m,2}({S}_{-\varepsilon},T)}+ \|q\|_{\WW^{m,2}(\Ga_{b_1})}+ \|p\|_{\WW^{m,2}(\Ga_{-\varepsilon})}).\label{V3.2}\eeq
\end{thm}

We define the linear operator $\L_0:$ $\WW^{m,2}({S}_{-\varepsilon},T)\rw\WW^{m,2}({S}_{-\varepsilon},T)$ by
\be\L_0V
=e^{-\g\kappa}[(\div Q\nabla\mathbf{n} V)QX+(\div V)\nabla\mathbf{n} X]-\C V\label{x3.1}\ee with domain
\beq D(\L_0)=\{\,V\in \WW^{m,2}(S,T)\big|&&\,\div V\in \WW^{m,2}({S}_{-\varepsilon}), \div Q\nabla\mathbf{n} V\in \WW^{m,2}({S}_{-\varepsilon})\nonumber\\
&&\,\<V,QX\>|_{\Ga_{-\varepsilon}}=\<V,\nu\>|_{\Ga_{b_1}}=0\,\},\nonumber\eeq
where
\be\C V=\eta e^{-\ga\kappa}(\<V,QX\>QX+\<V,\nabla\mathbf{n} X\>\nabla\mathbf{n} X),\quad \eta(x)=\varsigma\eta_0(s)\label{n3.1}\ee  for $x=\a(t,s)\in{S}_{-\varepsilon}.$
Moreover, set
\be\L_0^*V=\nabla\mathbf{n} Q D(e^{-\g\kappa}\<QX,V\>)-D(e^{-\g\kappa}\<\nabla\mathbf{n} X,V\>)-\C V.\label{n3.2}\ee The domain of \(\mathcal{L}_0^*\) is
$$ D(\L_0^*)=\{\,V\in \WW^{m,2}(S,T)\big|\,\L_0^*V\in\WW^{m,2}({S}_{-\varepsilon},T),\,\<V,\nu\>|_{\Ga_{-\varepsilon}}=\<V,QX\>|_{\Ga_{b_1}}=0\,\}.$$

By Theorem \ref{3t.1} and \cite[Theorems 3.3 and 3.4]{CY1} we have the following.
\begin{thm}\label{3t.2}
$(i)$\,\,\,The operators $\L_0$ and $\L_0^*$ have bounded inverses on $\WW^{m,2}({S}_{-\varepsilon},T)$ with the estimates
\beq&&\|\L_0^{-1}V\|_{\WW^{m,2}({S}_{-\varepsilon},T)}+\|\L_0^{-1}V\|_{\WW^{m,2}(\pl{S}_{-\varepsilon}, T)}+\nonumber\\
&&\|{\L_0^*}^{-1}V\|_{\WW^{m,2}({S}_{-\varepsilon},T)}+\|{\L_0^*}^{-1}V\|_{\WW^{m,2}(\pl{S}_{-\varepsilon}, T)}\leq C\|V\|_{\WW^{m,2}({S}_{-\varepsilon},T)}\eeq for all $V\in\WW^{m,2}({S}_{-\varepsilon},T).$

$(ii)$\,\,\,The operators $\L^{-1}_0\T$ and ${\L_0^*}^{-1}\T:$ $\WW^{m,2}({S}_{-\varepsilon},T)\rw\WW^{m+1,2}({S}_{-\varepsilon},T)$ are bounded, where $\T\in C^2({S}_{-\varepsilon},T^2)$ is a given linear transformation field with support in $S_{\varepsilon/2}=\{\,\a(t,s)\,|\,(t,s)\in[0,a)\times(\varepsilon/2,b_1)\,\}$. 
\end{thm}

Define the operators $\L$ and $\L^*:$ $\D'_m(S,T)\rightarrow \D'_{m+1}(S,T)$ by
$$\L V=\L_0 V+\C V\quad\mbox{and}\quad \L^*V=\L_0^*V+\C V,$$ respectively.
Given $Z_0\in \E_m({S}_{-\varepsilon},T)$, consider the problem
\be
\begin{cases}
\L^*Z+\nabla\mathbf{n} Dz=Z_0\qflq x\in{S}_{-\varepsilon},\\
-\Delta z=e^{-\g\kappa}\Pi(X,Z)\tr\Pi\qflq x\in{S}_{-\varepsilon},\\
\pl_\nu z\big|_{\pl{S}_{-\varepsilon}}=
\<Z,\nu\>\big|_{\Ga_{-\varepsilon}}=\<Z,QX\>\big|_{\Ga_{b_1}}=0.
\end{cases}\label{1.33}
\ee

We define the operator $\A:$ $\LL^2({S},T)\rw\WW^{1,2}({S},T)$ by
$$\A Z=\nabla\mathbf{n} Dz\qflq Z\in\LL^2({S},T),$$ where $z\in \WW^{2,2}({S},T)$ is the solution to 
$$\left\{\begin{array}{l}-\Delta z=e^{-\g\kappa}\<Z,\nabla\mathbf{n} X\>\tr\Pi\qflq x\in{S},\\
\pl_\nu z=0\quad\mbox{on}\quad x\in\Ga.\end{array}\right.$$ Then
$\A:$ $\LL^2({S},T)\rw\LL^2({S}_{-\varepsilon},T)$ is compact. By Theorem \ref{3t.2}, it follows that
$${\L_0^*}^{-1}(\C+\A):\quad \LL^2({S},T)\rw\LL^2({S}_{-\varepsilon},T)$$ is also compact. Define the subspaces
$$\N=\{\,W\in\LL^2({S}_{-\varepsilon},T)\,|\,W+{\L_0^*}^{-1}(\C+\A)W=0\,\},$$
$$\N_*=\{\,W\in\LL^2({S}_{-\varepsilon},T)\,|\,W+\L_0^{-1}(\C+\A^*)W=0\,\}.$$ Both \(\mathcal{N}\) and \(\mathcal{N}_*\) have finite dimension, i.e., \(\dim \mathcal{N} < \infty\) and \(\dim \mathcal{N}_* < \infty\).

\begin{pro}\label{p3.1}  Let $S$ be of class $\CC^{m+3}.$
Problem $(\ref{1.33})$ admits a unique solution $Z\perp \N$ in $\LL^2({S}_{-\varepsilon},T)$ if and only if  $Z_0\perp \N_*$ in $\LL^2({S}_{-\varepsilon},T).$ Moreover, there exists $C>0$ such that if
$Z_0\in\WW^{2,2}({S}_{-\varepsilon},T)$ and $(Z,z)$ solves problem $(\ref{1.33})$ with $Z\perp \N,$  then $(Z,z)\in\WW^{m,2}({S}_{-\varepsilon},T)\times\WW^{m+2,2}({S}_{-\varepsilon})$ satisfying
\be\|Z\|^2_{\WW^{m,2}({S}_{-\varepsilon},T)}+\|z\|^2_{\WW^{m+2,2}({S}_{-\varepsilon})}+\|Z\|^2_{\WW^{m,2}(\pl{S}_{-\varepsilon},T)}\leq C\|Z_0\|^2_{\WW^{m,2}({S}_{-\varepsilon},T)}.\label{3.2}\ee
\end{pro}

\begin{proof} For a given $Z_0\in\LL^2({S}_{-\varepsilon},T),$ it is straightforward to verify that $Z\in\LL^2({S}_{-\varepsilon},T)$ solves problem (\ref{1.33}) if and only if $Z$ satisfies
$$Z+{\L_0^*}^{-1}(\C+\A)Z={\L_0^*}^{-1}Z_0\qiq\LL^2({S}_{-\varepsilon},T).$$
By Fredholm's theorem, this problem has a solution $Z\in\LL^2({S}_{-\varepsilon},T)$ if and only if $Z_0\perp \N_*$ in $\LL^2({S}_{-\varepsilon},T).$

New, let $Z_0\in\WW^{m,2}({S}_{-\varepsilon},T).$ By \cite[Theorems 3.1 and 3.2]{CY1}, we have  $Z\in\WW^{m,2}({S}_{-\varepsilon},T).$ By the regularity theory for elliptic problems,  $z\in\WW^{m+2,2}({S}_{-\varepsilon})$ and the estimate
$$\|z\|^2_{\WW^{m+2,2}({S}_{-\varepsilon})}\leq C\|Z\|^2_{\WW^{m,2}({S}_{-\varepsilon},T)}$$ holds.
 Additionally, using \cite[Lemma 3.2 and Theorem 2.2]{CY1}, we obtain
$$\|Z\|^2_{\WW^{m,2}({S}_{-\varepsilon},T)}+\|Z\|^2_{\WW^{m,2}(\pl{S}_{-\varepsilon},T)}\leq C\|Z_0\|^2_{\WW^{m,2}({S}_{-\varepsilon},T)}.$$ Combining these results yields \eqref{3.2}. 
\end{proof}

Let $\{\Phi_1, \cdots, \Phi_k\}$ be an orthonormal basis of $\N,$ where $k=\dim\N.$ Define the projection operator ${\P}:$ $[\WW^{m,2}({S}_{-\varepsilon},T)]'\rw\N$ by
\beq
{\P}V=\sum_{j=1}^k(V,\Phi_j)_{\LL^2({S}_{-\varepsilon},T)}\Phi_k\qflq V\in[\WW^{m,2}({S}_{-\varepsilon},T)]'.\label{defP}
\eeq

\begin{thm}\label{t3.3} Let $S$ be of class $\CC^5.$ Then there exists a constant $C>0$ such that
\beq&&\|V\|_{[\WW^{2,2}({S}_{-\varepsilon},T)]'}+\|v\|_{[\WW^{1,2}({S}_{-\varepsilon})]'}\leq C\Big(\|{\P}V\|_{[\WW^{2,2}({S}_{-\varepsilon},T)]'}+I(F,f,p,q)\Big),\quad\label{3x.25}\eeq 
for all $(v,V)\in \WW^{2,2}(S_{-\varepsilon})\times \WW^{2,2}(S_{-\varepsilon},T)$, 
where
 \be \left\{\begin{array}{l}
F=Dv-\nabla\mathbf{n} V,\\
f=\div V+(\tr\Pi)v,\\
p=\<V,QX\>|_{\Ga_{-\varepsilon}},\quad q=\<V,\nu\>|_{\Ga_{b_1}},\end{array}\right.\label{3.11n}\ee
and
\beq I(F,f,p,q)&&=\|\div QF\|_{[\WW^{2,2}({S}_{-\varepsilon})]'}+\|F\|_{[\WW^{2,2}({S}_{-\varepsilon},T)]'}+\|f\|_{[\WW^{2,2}({S}_{-\varepsilon})]'}\nonumber\\
&&\quad
+\|p\|_{\WW^{-2,2}(\Ga_{-\varepsilon})}+\|q\|_{\WW^{-2,2}(\Ga_{b_1})}.\nonumber\eeq
\end{thm}

\begin{proof} Let $(v,V)$ be of $C^2$ on $S_{-\varepsilon}$. From (\ref{3.11n}) we derive
\be \left\{\begin{array}{l}\div Q\nabla \n V=-\div QF,\\
\div V=-\rho v+f,\end{array}\right.\label{3x.26}\ee where $\rho=\tr\Pi.$
Combining \eqref{x3.1}, \eqref{n3.1}, and \eqref{3x.26}, we obtain
\be\L_0V+\C V=G,\label{3x.27}\ee where
$$G=-e^{-\g\kappa}[(\div QF)QX+(\rho v-f)\nabla\mathbf{n} X].$$

Define
\be\Xi=\{\, Z\in\WW^{2,2}({S}_{-\varepsilon},T)\,|\,Z\perp\N_*\,\mbox{in $\LL^2({S}_{-\varepsilon},T)$\,}\}.\label{3n.27}\ee For a given $Z_0\in\Xi,$ solving problem (\ref{1.33}) yields a solution $(Z, z)\in\WW^{2,2}({S}_{-\varepsilon},T)\times\WW^{4,2}({S}_{-\varepsilon}).$ Let
 $$z_1=e^{-\g\kappa}\<Z,QX\>,\quad z_2=\rho e^{-\g\kappa}\<Z,\nabla\mathbf{n} X\>,\quad z_3=e^{-\g\kappa}\<Z,\nabla\mathbf{n} X\>.$$ By (\ref{1.33}), (\ref{3x.7}), and the first equation in
 (\ref{3.11n}), we  have
 \beq(V,Z_0)_{\LL^2({S}_{-\varepsilon},T)}&&=(V,\L_0^*Z+\C Z+\nabla\mathbf{n} Dz)_{\LL^2({S}_{-\varepsilon},T)}=(G, Z)_{\LL^2({S}_{-\varepsilon},T)}\nonumber\\
 &&\quad+(Dv-F,Dz)_{\LL^2({S}_{-\varepsilon},T)}+\int_{\Ga_{-\varepsilon}}\pounds_2p\<Z,QX\>d\Ga\nonumber\\
 &&\quad-\int_{\Ga_{b_1}}\pounds_1q\<Z,\nu\>d\Ga\nonumber\\
 &&=-(\div QF,z_1)_{\LL^2({S}_{-\varepsilon})}+(f, z_3)_{\LL^2({S}_{-\varepsilon})}\nonumber\\
 &&\quad-(F,Dz)_{\LL^2({S}_{-\varepsilon},T)}-(v,z_2)_{\LL^2({S}_{-\varepsilon})}+(Dv,Dz)_{\LL^2({S}_{-\varepsilon})}\nonumber\\
 &&\quad+\int_{\Ga_{-\varepsilon}}\pounds_2p\<Z,QX\>d\Ga-\int_{\Ga_{b_1}}\pounds_1q\<Z,\nu\>d\Ga.\nonumber\eeq
Noting that
$$-(v,z_2)_{\LL^2({S}_{-\varepsilon})}+(Dv,Dz)_{\LL^2({S}_{-\varepsilon})}=0,$$
by (\ref{3.2}) with $m=2$ and Lemma \ref{3l.3},  we obtain
\beq&&|(V,Z_0)_{\LL^2({S}_{-\varepsilon},T)}|\leq\|\div QF\|_{[\WW^{2,2}({S}_{-\varepsilon})]'}\|z_1\|_{\WW^{2,2}({S}_{-\varepsilon})}\nonumber\\
&&\quad+\|f\|_{[\WW^{2,2}({S}_{-\varepsilon})]'}\|z_3\|_{\WW^{2,2}({S}_{-\varepsilon})}+\|F\|_{[\WW^{2,2}({S}_{-\varepsilon},T)]'}\|Dz\|_{\WW^{2,2}({S}_{-\varepsilon},T)}\nonumber\\
&&\quad+C(\|p\|_{\WW^{-2,2}(\Ga_{-\varepsilon})}+\|q\|_{\WW^{-2,2}(\Ga_{b_1})})\|Z\|_{\WW^{2,2}(\pl{S}_{-\varepsilon},T)}\nonumber\\
&&\leq CI(F,f,p,q)\|Z_0\|_{\WW^{2,2}({S}_{-\varepsilon},T)}\qflq Z_0\in\Xi.\label{n3x.28}\eeq

Since $\Xi$ is a closed subspace of $\WW^{2,2}({S}_{-\varepsilon},T),$ (\ref{n3x.28}) implies the existence of an element $\B V\in\Xi$ such that
$$(V,Z_0)_{\LL^2({S}_{-\varepsilon},T)}=({\B}V,Z_0)_{\WW^{2,2}({S}_{-\varepsilon},T)}\qflq Z_0\in\Xi.$$

Let $\II$ be the canonical isomorphism from  $[\WW^{2,2}({S}_{-\varepsilon},T)]'\rw\WW^{2,2}({S}_{-\varepsilon},T)$ and set $V_0={\II}^{-1}{\B}V.$ Then $V_0\in[\WW^{2,2}({S}_{-\varepsilon},T)]'$ satisfies
$$(V-V_0,Z_0)_{\LL^2({S}_{-\varepsilon},T)}=0\qflq Z_0\in\Xi,$$ which implies
\be V-V_0\in\N_*.\label{3x.28}\ee Therefore,
$$V={\P} V+V_0\in[\WW^{2,2}({S}_{-\varepsilon},T)]',$$ where ${\P}V=V-V_0.$ By (\ref{n3x.28}), we derive
\beq\|V\|_{[\WW^{2,2}({S}_{-\varepsilon},T)]'}&&\leq\|\mathcal{P} V\|_{[\WW^{2,2}({S}_{-\varepsilon},T)]'}+\|V_0\|_{[\WW^{2,2}({S}_{-\varepsilon},T)]'}\nonumber\\
&&=\|\mathcal{P} V\|_{[\WW^{2,2}({S}_{-\varepsilon},T)]'}+\sup_{Z_0\in\WW^{2,2}({S}_{-\varepsilon},T), \|Z_0\|_{\WW^{2,2}({S}_{-\varepsilon},T)}=1}}(V,Z_0)_{\LL^2({S}_{-\varepsilon},T)\nonumber\\
&&\leq \|\mathcal{P} V\|_{[\WW^{2,2}({S}_{-\varepsilon},T)]'}+CI(F,f,p,q).\label{3x.29}\eeq

For a given $z_0\in\WW^{1,2}({S}_{-\varepsilon}),$ solve problem
$$\Delta w=z_0\qflq x\in{S}_{-\varepsilon},\quad \<Dw,\nu\>|_{\pl{S}_{-\varepsilon}}=0.$$ Then
$$\|w\|_{\WW^{3,2}({S}_{-\varepsilon})}\leq C\|z_0\|_{\WW^{1,2}({S}_{-\varepsilon})}\qflq z_0\in\WW^{1,2}({S}_{-\varepsilon}).$$
By the first equation in (\ref{3.11n}), 
\beq(v,z_0)_{\LL^2({S}_{-\varepsilon})}&&=-(Dv,Dw)_{\LL^2({S}_{-\varepsilon},T)}=-(\nabla\mathbf{n} V+F,Dw)_{\LL^2({S}_{-\varepsilon},T)}\nonumber\\
&&\leq C(\|V\|_{[\WW^{2,2}({S}_{-\varepsilon},T)]'}+\|F\|_{[\WW^{2,2}({S}_{-\varepsilon},T)]'})\|z_0\|_{\WW^{1,2}({S}_{-\varepsilon})}\nonumber\eeq for all $z_0\in\WW^{1,2}({S}_{-\varepsilon}),$ which yields
\be \|v\|_{[\WW^{1,2}({S}_{-\varepsilon})]'}\leq C(\|V\|_{[\WW^{2,2}({S}_{-\varepsilon},T)]'}+\|F\|_{[\WW^{2,2}({S}_{-\varepsilon},T)]'}).\label{n3x.31}\ee

Combining \eqref{3x.29} and \eqref{n3x.31} gives \eqref{3x.25} by density.
\end{proof}
We need the following uniqueness result.
\begin{lem}\label{lunique} Let $\b:$ $[0,1]\rw M$ be a non-characteristic curve, i.e, 
\be\Pi(\dot\b(t),\dot\b(t))\not=0\qfq t\in[0,1].\label{unique1}\ee Let  $U\subset M$ be a neighborhood of $\b$ such that $\kappa<0$ on $U.$
Suppose $v\in C^2(U)$ solves problem
\be\left\{\begin{array}{l}\div(\nabla\mathbf{n})^{-1}Dv=f_0 v\qflq p\in U,\\
v=Dv=0\qoq \b,\end{array}\right.\label{unique}\ee where $f_0\in C(U)$ is given. Then for each $p_0$ on $\b,$ there is a neighborhood $V\subset U$ of $p_0$ such that
$$v=0\qfq p\in V.$$ 
\end{lem}
\begin{proof}{\bf Step 1.}\,\,\,Fix $p\in U$ such that $\kappa(p)<0.$ Let $\{e_1,e_2\}$ be an orthonormal basis of $T_pM$ such that
$$\nabla_{e_i}\n=\lam_ie_i\qaq p\qfq i=1,\,2,$$ where $\lam_i$ are the principal curvatures. Let $\{E_1,E_2\}$ be an orthonormal frame around $p$ such that $E_i(p)=e_i.$ We compute at $p$
\beq\div(\nabla\mathbf{n})^{-1}Dv&&=E_1\<Dv,(\nabla\mathbf{n})^{-1}E_1\>+E_2\<Dv,(\nabla\mathbf{n})^{-1}E_2\>\nonumber\\
&&=\<D_{E_1}Dv,(\nabla\mathbf{n})^{-1}E_1\>+\<D_{E_2}Dv,(\nabla\mathbf{n})^{-1}E_2\>\nonumber\\
&&\quad+\mbox{the first order derivative terms}\nonumber\\
&&=\frac1{\lam_1}D^2v(E_1,E_1)+\frac1{\lam_2}D^2v(E_2,E_2)\nonumber\\
&&\quad+\mbox{the first order derivative terms}\qaq p.\nonumber\eeq It follows that
\beq\kappa\div(\nabla\mathbf{n})^{-1}Dv&&=\<D^2v,Q^*\Pi\>+X(v)\qfq p\in U,\nonumber\eeq where $Q$ is defined by (\ref{n2.32}) and $X$ is a vector field on $U.$

{\bf Step 2.}\,\,\,We shall solve (\ref{unique}) locally in asymptotic coordinate systems. Let $p_0=\b(t_0)$ for some $t_0\in(0,1).$ Let $(U_0,\psi)$ be an asymptotic coordinate system around $p_0,$ where $p_0\in U_0\subset U.$ Then
$$\Pi(\pl x_1,\pl x_1)=\Pi(\pl x_2,\pl x_2)=0,\quad \Pi^2(\pl x_1,\pl x_2)=-\kappa\det G>0\qfq p\in U_0,$$ where
$\det G=|\pl x_1\wedge\pl x_2|^2.$ Let $\psi(\b(t_0+t))=(\g_1(t),\g_2(t))$ for $t\in(-\varepsilon,\varepsilon),$ where $\varepsilon>0$ is given small. Then (\ref{unique1}) impies
$$\Pi(\dot\b(t_0+t),\dot\b(t_0+t))=2\dot\g_1(t)\dot\g_2(t)\Pi(\pl x_1,\pl x_2)\not=0,$$ that is, $\dot\g_1(t)\dot\g_2(t)\not=0$ for $t\in(-\varepsilon,\varepsilon).$ 

Define $w(x)=v\circ\psi^{-1}(x)$ for $x=(x_1,x_2)\in\R^2.$ By Step 1 and \cite[Proposition 4.1]{Yao2017}, $w$ solves problem
\be\begin{cases}w_{x_1x_2}(x)=\hat X(w)+\hat f_0w\qfq x\in\R^2,\\
w(\g(t))=w_{x_1}(\g(t))=w_{x_2}(\g(t))=0\qfq t\in(-\varepsilon,\varepsilon),\end{cases}\label{unique2}\ee where $\g(t)=(\g_1(t),\g_2(t)),$ $\hat X$ is a vector field, and $\hat f_0$ is a function on $\R^2.$ Without loss of generality, 
we assume
$$\g_1'(t)>0,\quad\g_2'(t)<0\qfq t\in(-\varepsilon,\varepsilon).$$ The other cases can be handled similarly.
By \cite[Proposition 3.1]{Yao2017}, the solution $w$ to problem (\ref{unique2}) satisfies
$$w(x)=0\qfq x\in E(\g),$$ where
$$E(\g)=\{\,(x_1,x_2)\in\R^2\,|\,\g_1(t)\leq x_1\leq\g_1(-\varepsilon),\,\,\g_2(-\varepsilon)\leq x_2\leq\g_2(\varepsilon),\,\,t\in(-\varepsilon,\varepsilon)\,\}.$$
A similar argument as for \cite[Proposition 3.1]{Yao2017} yields
$$w(x)=0\qfq x\in\hat E(\g),$$ where
$$\hat E(\g)=\{\,(x_1,x_2)\in\R^2\,|\,\g_1(\varepsilon)\leq x_1\leq\g_1(t),\,\,\g_2(-\varepsilon)\leq x_2\leq\g_2(\varepsilon),\,\,t\in(-\varepsilon,\varepsilon)\,\}.$$ 
Since the set of interior points of $E(\g)\cup\hat E(\g)$ is a neighborhood of $\psi(\b(t_0)),$  the set $U_0$ of interior points of
$\psi^{-1}[E(\g)\cup\hat E(\g)]\subset M$ is a neighborhood of $\b(t_0)$ such that $v(p)=0$ for $p\in U_0.$
\end{proof}

\begin{thm}\label{t3.4} Let $S$ be of class $\CC^5.$ Then there exists $C>0$ such that
\beq
&&\|W\|_{\LL^2({S}_{-\varepsilon},T)}
+\|w\|_{[\WW^{1,2}({S}_{-\varepsilon})]'}
+\|\div QW\|_{[\WW^{1,2}({S}_{-\varepsilon})]'}
+\|Dw-\nabla\n W\|_{[\WW^{2,2}({S}_{-\varepsilon},T)]'}
\nonumber\\
\leq &&C\Big(\|\sym DW+w\Pi\|_{\LL^2({S}_{-\varepsilon},T^2)}+\|W\|_{\LL^2(\pl{S}_{-\varepsilon},T)}
\nonumber\\
&&\qquad\qquad+\|w\|_{\WW^{-1,2}(\pl S_{-\varepsilon})}
+\|Dw-\nabla\n W\|_{\WW^{-2,2}(\Ga_{-\varepsilon})}
\Big)\quad\qquad\label{3.26}\eeq
holds for all $y=W+w\mathbf{n}\in \WW^{2,2}(S,\R^3)$.
\end{thm}
\begin{proof}Given $y=W+w\n\in C^3(S,\R^3)$, let $U=\sym DW+w\Pi$.  By Theorem \ref{t2.1}, pairs $(w,W)$ and \((v, V)\) satisfy the following systems pointwise on ${S}_{-\varepsilon}$
\be v=\frac12\div QW,\label{n3xx.33}\ee
\be\left\{\begin{array}{l}Dw=\nabla\mathbf{n} W-QV,\\
\div W=-\rho w+\tr U,\end{array}\right.\label{3n.34}\ee and
\be\left\{\begin{array}{l}Dv=\nabla\mathbf{n} V+Q\div QUQ,\\
\div V=-\rho v-\<Q\nabla\mathbf{n},U\>.\end{array}\right.\label{n3x.34}\ee 
 It follows that
$$\div Q\nabla\mathbf{n} W=-\div V=\rho v-\<Q\nabla\mathbf{n},U\>.$$
Lemma \ref{l3.1} implies
\beq\|W\|_{\LL^2({S}_{-\varepsilon},T)}&&\leq C(\|\sym DW\|_{[\WW^{1,2}({S}_{-\varepsilon},T^2)]'}+\|W\|_{\LL^2(\pl{S}_{-\varepsilon},T)})\nonumber\\
&&\leq C(\|U\|_{[\WW^{1,2}({S}_{-\varepsilon},T^2)]'}+\|w\|_{[\WW^{1,2}({S}_{-\varepsilon})]'}+\|W\|_{\LL^2(\pl{S}_{-\varepsilon},T)}).\label{l3.42}
\eeq
Applying Theorem \ref{t3.3} to problem (\ref{3n.34}) with
$$F=-QV,\quad f=\tr U,\quad p=\<W,QX\>|_{\Ga_{-\varepsilon}},\quad q=\<W,\nu\>|_{\Ga_{b_1}},$$ we obtain
\beq
\|w\|_{[\WW^{1,2}({S}_{-\varepsilon})]'}&&\leq C\big(\|{\P}W\|_{[\WW^{2,2}({S}_{-\varepsilon},T)]'}+\|\div V\|_{[\WW^{2,2}({S}_{-\varepsilon})]'}+\|V\|_{[\WW^{2,2}({S}_{-\varepsilon},T)]'}\nonumber\\
&&\quad
+\|\tr U\|_{[\WW^{2,2}({S}_{-\varepsilon})]'}
+\|\<W,QX\>\|_{\WW^{-2,2}(\Ga_{-\varepsilon})}+\|\<W,\nu\>\|_{\WW^{-2,2}(\Ga_{b_1})}\big).\qquad\quad
\label{3x.35}\eeq
Applying Theorem \ref{t3.3} to problem (\ref{n3x.34}) with
$$F=Q\div QUQ,\quad f=-\<Q\nabla\mathbf{n},U\>,\quad p=\<V,QX\>,\quad q=\<V,\nu\>,$$
 yields
\beq&&\|\div V\|_{[\WW^{2,2}({S}_{-\varepsilon})]'}+\|V\|_{[\WW^{2,2}({S}_{-\varepsilon},T)]'}\nonumber\\
&&\leq \|\<Q\nabla\mathbf{n},U\>\|_{[\WW^{2,2}({S}_{-\varepsilon})]'}+C(\|V\|_{[\WW^{2,2}({S}_{-\varepsilon},T)]'}+\|v\|_{[\WW^{1,2}({S}_{-\varepsilon})]'})\nonumber\\
&&\leq C(\|{\P}V\|_{[\WW^{2,2}({S}_{-\varepsilon},T)]'}+\|U\|_{\LL^2({S}_{-\varepsilon},T^2)}\nonumber\\
&&\quad+\|\<V,QX\>\|_{\WW^{-2,2}(\Ga_{-\varepsilon})}+\|\<V,\nu\>\|_{\WW^{-2,2}(\Ga_{b_1})}),\label{3x.36}\eeq where Lemma \ref{3l.3} is used. 
By (\ref{n3xx.33}) and the first equation in \eqref{3n.34}, we have
\beq
&&\|v\|_{[\WW^{1,2}({S}_{-\varepsilon})]'}
+\|V\|_{[\WW^{2,2}({S}_{-\varepsilon},T)]'}
\nonumber\\
\le&& C\big(\|W\|_{\LL^2({S}_{-\varepsilon},T)}
+\|w\|_{[\WW^{1,2}({S}_{-\varepsilon})]'}
+\|W\|_{\LL^2(\pl{S}_{-\varepsilon},T)}
+\|w\|_{\WW^{-1,2}(\pl{S}_{-\varepsilon})}
\big),
\label{v0V0}
\eeq
and 
\beq
\|\<V,\nu\>\|_{\WW^{-2,2}(\Ga_{b_1})}
\le C\big(\|w\|_{\WW^{-1,2}(\Ga_{b_1})}
+\|W\|_{\WW^{-2,2}(\Ga_{b_1})}\big),
\label{tangentGa}
\eeq
since $Q\nu|_{\Ga_{b_1}}$ is tangent to curve $\Ga_{b_1}$. 
Substituting \eqref{3x.35}--(\ref{tangentGa}) into \eqref{l3.42}, we derive
\beq
&&\|W\|_{\LL^2({S}_{-\varepsilon},T)}
+\|w\|_{[\WW^{1,2}({S}_{-\varepsilon})]'}
+\|v\|_{[\WW^{1,2}({S}_{-\varepsilon})]'}
+\|V\|_{[\WW^{2,2}({S}_{-\varepsilon},T)]'}
\nonumber\\
\leq &&C\Big(\|U\|_{\LL^2({S}_{-\varepsilon},T^2)}+\|W\|_{\LL^2(\pl{S}_{-\varepsilon},T)}
+\|w\|_{\WW^{-1,2}(\pl{S}_{-\varepsilon})}
+\|\<V,QX\>\|_{\WW^{-2,2}(\Ga_{-\varepsilon})}
\nonumber\\
&&\quad+\|{\P}W\|_{[\WW^{2,2}({S}_{-\varepsilon},T)]'}
+\|{\P}V\|_{[\WW^{2,2}({S}_{-\varepsilon},T)]'}
\Big).\quad\label{n3x.37}\eeq
Let $\Y(S)$ be the closure of $C^3(S,\R^3)$ under norm 
\beq
\|y\|^2_{\Y(S)}
=&&\|W\|^2_{\LL^2({S}_{-\varepsilon},T)}
+\|w\|^2_{[\WW^{1,2}({S}_{-\varepsilon})]'}
+\|v\|^2_{[\WW^{1,2}({S}_{-\varepsilon})]'}
+\|V\|^2_{[\WW^{2,2}({S}_{-\varepsilon},T)]'}
\nonumber\\
&&+\|U\|^2_{\LL^2({S}_{-\varepsilon},T^2)}
+\|W\|^2_{\LL^2(\pl{S}_{-\varepsilon},T)}
+\|w\|^2_{\WW^{-1,2}(\pl{S}_{-\varepsilon})}
+\|V\|^2_{\WW^{-2,2}(\Ga_{-\varepsilon})}
,
\nonumber
\eeq
where norms $\|{\P}W\|_{[\WW^{2,2}({S}_{-\varepsilon},T)]'}$ and $\|{\P}Q(Dw-\nabla\mathbf{n} W)\|_{[\WW^{2,2}({S}_{-\varepsilon},T)]'}$ can be omitted by the definition (\ref{defP}) of $\P$. By density, (\ref{n3x.37}) holds for all $y\in\Y(S)$. 

We claim that the terms $\|{\P}W\|_{[\WW^{2,2}({S}_{-\varepsilon},T)]'}$ and $\|{\P}Q(Dw-\nabla\mathbf{n} W)\|_{[\WW^{2,2}({S}_{-\varepsilon},T)]'}$ can be removed from (\ref{n3x.37}) via a compactness-uniqueness argument to obtain (\ref{3.26}). 
 Assume for contradiction that it fails. Then there exist sequences $y_n=W_n+w_n\n\in\Y(S)$ such that $w_n,$ $W_n,$ $v_n,$ $V_n,$ and $U_n$ satisfy \eqref{n3xx.33}--\eqref{n3x.34} in the sense of generalized tensors and 
\beq
1=&&\|W_n\|_{\LL^2({S}_{-\varepsilon},T)}+\|w_n\|_{[\WW^{1,2}(S_{-\varepsilon})]'}
+\|v_n\|_{[\WW^{1,2}({S}_{-\varepsilon})]'}
+\|V_n\|_{[\WW^{2,2}({S}_{-\varepsilon},T)]'}
\nonumber\\
\geq&& n\big(\|U_n\|_{\LL^2({S}_{-\varepsilon},T^2)}+\|W_n\|_{\LL^2(\pl{S}_{-\varepsilon},T)}
+\|w_n\|_{\WW^{-1,2}(\pl{S}_{-\varepsilon})}
+\|V_n\|_{\WW^{-2,2}(\Ga_{-\varepsilon})}\big)
\qquad\label{l3.45}\eeq
 for all $n\in\mathbb{N}^+.$
By (\ref{defP}), let 
\[
{\P}W_n=\sum_{j=1}^ka_{nj}\Phi_j\quad\mbox{and}\quad
{\P}V_n=\sum_{j=1}^kb_{nj}\Phi_j.
\]
 Then $\{(a_{n1},\dots,a_{nk})\}_{n=1}^\infty$ and $\{(b_{n1},\dots,b_{nk})\}_{n=1}^\infty$ is bounded in $\R^k$ by (\ref{defP}) and (\ref{l3.45}).  Extracting a subsequence not relabeled such that $\{(a_{n1},\dots,a_{nk})\}_{n=1}^\infty$ and $\{(b_{n1},\dots,b_{nk})\}_{n=1}^\infty$ converge in $\R^k$, we obtain $y_n$ is a Cauchy series in $\Y(S)$ by (\ref{n3x.37}). Hence, $y_n\to y_0$ in $\Y(S)$ for some $y_0=W_0+w_0\n\in\Y(S)$, as $n\to \infty$. By density, there exists a Cauchy series still denoted by $y_n=W_n+w_n\n\in C^3(S_{-\varepsilon},\R^3)$ converging to $y_0$ in $\Y(S)$. Therefore, there exists $v_0\in [\WW^{1,2}(S_{-\varepsilon})]'$, $V_0\in [\WW^{2,2}(S_{-\varepsilon},T)]'$ such that 
\beq
&&
\|W_n-W_0\|^2_{\LL^2({S}_{-\varepsilon},T)}
+\|w_n-w_0\|^2_{[\WW^{1,2}({S}_{-\varepsilon})]'}
+\|\div QW_n-2v_0\|_{[\WW^{1,2}({S}_{-\varepsilon})]'}
\nonumber\\
&&\quad+\|Dw_n-\nabla\n W_n-QV_0\|_{[\WW^{2,2}({S}_{-\varepsilon},T)]'}
+\|\sym DW_n+w_n\Pi\|_{\LL^2({S}_{-\varepsilon},T^2)}
\nonumber\\
&&\qquad\quad
+\|W_n\|_{\LL^2(\pl{S}_{-\varepsilon},T)}
+\|w_n\|_{\WW^{-1,2}(\pl{S}_{-\varepsilon})}
+\|Dw_n-\nabla\n W_n\|_{\WW^{-2,2}(\Ga_{-\varepsilon})}
\to 0,\qquad\qquad
\label{Cauchy}
\eeq
as $n\to \infty$. By Theorem \ref{t2.1}, the following holds on $S_{-\varepsilon}$ in the sense of generalized tensors 
$$v_0=\frac12\div QW_0,\qquad\sym DW_0+w_0\Pi=0,$$ 
\be \left\{\begin{array}{l}
Dw_0=\nabla\mathbf{n} W_0-QV_0,\\
\div W_0=-\rho w_0,\end{array}\right.\label{n3x.41}\ee
and
\be\left\{\begin{array}{l}Dv_0=\nabla\mathbf{n} V_0,\\
\div V_0=-\rho v_0.\end{array}\right.\label{n3x.40}\ee

To derive contradictions, we first derive an equation about $(v_0,V_0)$ holding in weak sense, then improve the regularity of $(w_0,W_0,v_0,V_0)$ and finally apply the uniqueness of solutions to show that they must vanish. 

For consistency, we still use notation $(\cdot,\cdot)_{\LL^2(M)}$ to represent the acting of generalized tensors on $M$ on the corresponding testing functions. 
\\
{\bf Step 1.} Claim: $v_0\in [\WW^{1,2}(S_{-\varepsilon})]'$, $V_0\in [\WW^{2,2}(S_{-\varepsilon},T)]'$ satisfy 
$\L_0V_0=F_0$ on $S_{-\varepsilon}$ 
in weak sense, i.e. $(V_0,\L_0^*W)_{\LL^2(S_{-\varepsilon})}=(F_0,W)_{\LL^2(S_{-\varepsilon})}$ holds for all $W\in \WW^{2,2}(S_{-\varepsilon},T)$, where 
\[
F_0=-e^{-\g\kappa}\rho v_0\nabla\mathbf{n} X-\C V_0.
\]

To prove the claim, let $\tilde y_n=\tilde W_n+\tilde w_n\n$ be the extension of $y_n$ by zero outside $S_{-\varepsilon}$  and $\tilde v_n,\tilde V_n,\tilde U_n$ are generalized tensors on $M$ defined by (\ref{2.28}), (\ref{V2.29}) and (\ref{U2.31}), respectively. Then there holds on $M$ in the sense of generalized tensors 
\be \tilde v_n=\frac12\div Q\tilde W_n,\quad
\tilde U_n=\sym D\tilde W_n+\tilde w_n\Pi\label{a3.48}\ee
\be\left\{\begin{array}{l}D\tilde w_n=\nabla\mathbf{n} \tilde W_n
-Q\tilde V_n,\\
\div \tilde W_n=-\rho \tilde w_n+\tr \tilde U_n,\end{array}\right.\label{a3.34}\ee and
\be\left\{\begin{array}{l}D\tilde v_n
=\nabla\mathbf{n}\tilde V_n+Q\div Q\tilde U_nQ,\\
\div \tilde V_n=-\rho \tilde v_n-\<Q\nabla\mathbf{n},\tilde U_n\>.\end{array}\right.\label{a3.49}
\ee 
Fix $W\in C^3_0(M,T)$. 
Then by (\ref{x3.1}), (\ref{n3.2}) and the definitions of generalized derivatives, one can easily derive 
\beq
(\tilde\L_0\tilde V_n,W)_{\LL^2(M)}=(\tilde V_n,\L_0^*W)_{\LL^2(M)},\label{a3.52}
\eeq
where $\tilde\L_0$ on the left is the generalized version of operator $\L_0$. (\ref{a3.34}) yield 
\beq
(\tilde V_n,\L_0^*W)_{\LL^2(M)}
=&&(-D\tilde w_n+\nabla\mathbf{n} \tilde W_n,Q\L_0^*W)_{\LL^2(M)}
\nonumber\\
=&&(\tilde w_n,\div Q\L_0^*W)_{\LL^2(M)}
+(\nabla\mathbf{n} \tilde W_n,Q\L_0^*W)_{\LL^2(M)}
\nonumber\\
=&&\int_{S_{-\varepsilon}}w_n\div Q\L_0^*W
+\<\nabla\mathbf{n} W_n,Q\L_0^*W\>dg
\nonumber\\
=&&\int_{\partial S_{-\varepsilon}}w_n\<Q\L_0^*W,\nu\>
+\int_{S_{-\varepsilon}}\<Q(Dw_n-\nabla\mathbf{n} W_n),\L_0^*W\>dg.
\nonumber
\eeq
(\ref{Cauchy}) yields 
\beq
\lim_{n\to\infty}(\tilde V_n,\L_0^*W)_{\LL^2(M)}
=(V_0,\L_0^*W)_{\LL^2(S_{-\varepsilon})}.\label{a3.53}
\eeq
On the other hand, by (\ref{x3.1}), (\ref{a3.48}), (\ref{a3.34}) and (\ref{a3.49}), we have 
\beq
&&(\tilde\L_0\tilde V_n,W)_{\LL^2(M)}\nonumber\\
=&&(e^{-\g\kappa}[(\div \div Q\tilde U_nQ)QX
-(\frac12\rho \div Q\tilde W_n+\<Q\nabla\mathbf{n},\tilde U_n\>)\nabla\mathbf{n} X]-\C \tilde V_n,W)_{\LL^2(M)}
\nonumber\\
=&&(\sym D\tilde W_n+\tilde w_n\Pi,P)_{\LL^2(M)}
+\frac12(Q\tilde W_n,D\xi)_{\LL^2(M)}
+(D\tilde w_n-\nabla\mathbf{n} \tilde W_n,Q\C W)_{\LL^2(M)}
\nonumber\\
=&&
-(\tilde W_n,\div\sym P+\frac12QD\xi
+\nabla\mathbf{n}Q\C W)_{\LL^2(M)}
+(\tilde w_n,\<\Pi,P\>-\div Q\C W)_{\LL^2(M)}
\nonumber\\
=&&
\int_{S_{-\varepsilon}}
-\<W_n,\div\sym P+\frac12QD\xi
+\nabla\mathbf{n}Q\C W\>
+w_n(\<\Pi,P\>-\div Q\C W)dg
\nonumber\\
=&&
\int_{\pl S_{-\varepsilon}}
\<-(\sym P)W_n
+\frac12\xi QW_n
-w_nQ\C W,\nu\>d\Ga
\nonumber\\
&&
+\int_{S_{-\varepsilon}}
\<\sym DW_n+w_n\Pi,P\>
-\frac12\xi\div QW_n
-\<Q(Dw_n-\nabla\mathbf{n}W_n),\C W\>
dg,
\nonumber
\eeq
where 
\beq
P
=&&QD^2[e^{-\g\kappa}\<W,QX\>]Q-e^{-\g\kappa}\<W,\nabla\mathbf{n} X\>Q\nabla\mathbf{n},
\nonumber\\
\xi
=&&e^{-\g\kappa}\rho \<W,\nabla\mathbf{n} X\>.
\nonumber
\eeq
(\ref{Cauchy})  yields for all $W\in C^3_0(M,T)$, there holds 
\beq
\lim_{n\to\infty}(\tilde\L_0\tilde V_n,W)_{\LL^2(M)}
=&&-(v_0,\xi)_{L^2(S_{-\varepsilon})}
-(V_0,\C W)_{L^2(S_{-\varepsilon})}
=(F_0,W)_{L^2(S_{-\varepsilon})}.
\label{a3.54}
\eeq
Therefore, by $S\subset\subset M$ and density, we conclude from (\ref{a3.52})--(\ref{a3.54}) that 
\[
(V_0,\L_0^*W)_{\LL^2(S_{-\varepsilon})}
=(F_0,W)_{L^2(S_{-\varepsilon})}\quad\mbox{holds for all}\quad W\in \WW^{2,2}(S_{-\varepsilon},T).
\]
{\bf Step 2.} Claim: $(w_0,W_0,v_0,V_0)$ is of $C^2$ on $S_{-\varepsilon}$. 

For any $Z\in C^2(S_{-\varepsilon},T)$, Theorem \ref{3t.2} implies that $W={\L_0^*}^{-1}Z\in \WW^{2,2}(S_{-\varepsilon},T)$ satisfies $\<W,\nu\>|_{\Ga_{-\varepsilon}}=\<W,QX\>|_{\Ga_{b_1}}=0$ and 
\[\|W\|_{\WW^{1,2}({S}_{-\varepsilon},T)}\leq C\|Z\|_{\WW^{1,2}({S}_{-\varepsilon},T)}.
\]
By Theorem \ref{3t.2} again, $L^*_0(\C W)=\T_1 L^*_0W+\T_2 W=\T_1 Z+\T_2 W$ implies
\beq
\|\C W\|_{\WW^{2,2}({S}_{-\varepsilon},T)}
=&&\|(L^*_0)^{-1}(\T_1 Z+\T_2 W)\|_{\WW^{2,2}({S}_{-\varepsilon},T)}\nonumber\\
\le&& C\|\T_1 Z+\T_2 W\|_{\WW^{1,2}({S}_{-\varepsilon},T)}
\le C\|Z\|_{\WW^{1,2}({S}_{-\varepsilon},T)},
\nonumber
\eeq
where $\T_1,\T_2\in C^2(S_{-\varepsilon},T^2)$ are linear transformation fields with supports in $S_{\varepsilon/2}$. 
Then 
\beq 
(V_0,Z)_{\LL^2(S_{-\varepsilon})}
=&&(V_0,\L_0^*W)_{\LL^2(S_{-\varepsilon})}
=(F_0,W)_{\LL^2(S_{-\varepsilon})}
\nonumber\\
=&&-(v_0,e^{-\g\kappa}\rho \<W,\nabla\mathbf{n} X\>)_{\LL^2(S_{-\varepsilon})}
-(V_0,\C W)_{\LL^2(S_{-\varepsilon},T)}. 
\nonumber
\eeq
Therefore, for any $Z\in C^2(S_{-\varepsilon},T)$, we have 
\beq 
&&|(V_0,Z)_{\LL^2(S_{-\varepsilon})}|\nonumber\\
\le&&
C\|v_0\|_{[\WW^{1,2}({S}_{-\varepsilon},T)]'}
\|W\|_{\WW^{1,2}({S}_{-\varepsilon},T)}
+C\|V_0\|_{[\WW^{2,2}({S}_{-\varepsilon},T)]'}
\|\C W\|_{\WW^{2,2}({S}_{-\varepsilon},T)}
\nonumber\\
\le&&
C\big(\|v_0\|_{[\WW^{1,2}({S}_{-\varepsilon},T)]'}
+\|V_0\|_{[\WW^{2,2}({S}_{-\varepsilon},T)]'}\big)
\|Z\|_{\WW^{1,2}({S}_{-\varepsilon},T)}.
\nonumber
\eeq
By density, we obtain $V_0\in [\WW^{1,2}({S}_{-\varepsilon},T)]'$ and then $v_0\in \LL^2({S}_{-\varepsilon})$ by (\ref{n3x.40}). By bootstrapping, we conclude $(v_0,V_0)$ is of $C^2$ on $S_{-\varepsilon}$. 
So is $(w_0,W_0)$ by (\ref{n3x.41}) and a similar argument.  Furthermore, by (\ref{Cauchy}) and (\ref{tangentGa}), we have 
\[ W_0=0,\quad w_0=0\quad \mbox{at}\quad \partial S_{-\varepsilon},
\]
and 
\beq
V_0=0\quad \mbox{at}\quad \Ga_{-\varepsilon},\quad
\<V_0,\nu\>=0\quad \mbox{at}\quad \Ga_{b_1}.\label{V0bd}
\eeq
{\bf Step 3.} Uniqueness of $C^2$ solutions. 

On $\Ga_{b_1},$ let $\tau=\a_t/|\a_t|$ and $\nu=Q\tau.$ $\{\tau,\nu\}$  forms a positively oriented orthonormal basis. We derive 
\beq2v_0&&=\div QW_0=\<D_\nu W_0,\tau\>-\<D_\tau W_0,\nu\>=\<D_\nu W_0,\tau\>+\<D_\tau W_0,\nu\>\nonumber\\
&&=-2w_0\Pi(\tau,\nu)=0\qflq x\in\Ga_{b_1}.\nonumber\eeq
Note that $\<(\nabla\n)^{-1}\nu,\nu\>>0$ at $\Ga_{b_1}$. It follows from (\ref{V0bd}) and (\ref{n3x.40}) that
\be\left\{\begin{array}{l}\div(\nabla\mathbf{n})^{-1}Dv_0=-\rho v_0\qflq x\in {S}_{-\varepsilon},\\
v_0=Dv_0=0\qflq x\in\Ga_{b_1}.\end{array}\right.\label{3.43}\ee
Since the surface $S$ is assumed to be of class $C^5,$ the coefficients of the principal part to problem (\ref{3.43}) are of class $\CC^3.$ It follows from the uniqueness for elliptic problems in \cite{Ar} that 
$$v_0 = 0\qflq x\in S_{-\varepsilon}\quad\mbox{with}\quad \kappa(x) > 0.$$
Similarly,
$$v_0=0\qflq x\in\Ga_{-\varepsilon}.$$ Then the system (\ref{3.43}) becomes
$$\left\{\begin{array}{l}\div(\nabla\mathbf{n})^{-1}Dv_0=-\rho v_0\qflq x\in {S}_{-\varepsilon},\quad\kappa(x)<0,\\
v_0=Dv_0=0\qflq x\in\Ga_{-\varepsilon}.\end{array}\right.$$
Since the curves
$$\a(\cdot,s)\qflq -\varepsilon\leq s\leq 0,$$ are not characteristic, Lemma \ref{lunique} yields
$$v_0=0\qflq x\in{S}_{-\varepsilon},\quad \kappa(x)<0.$$
Thus,
$$v_0=V_0=0\qflq x\in{S}_{-\varepsilon}.$$ Moreover, by (\ref{n3x.41}), we derive
$$\left\{\begin{array}{l}\div(\nabla\mathbf{n})^{-1}Dw_0=-\rho w_0\qflq x\in {S}_{-\varepsilon},\\
w_0=Dw_0=0\qflq x\in\pl{S}_{-\varepsilon}.\end{array}\right.$$
By a similar argument as for $(v_0,V_0),$ we obtain
$$w_0=W_0=0\qflq x\in{S}_{-\varepsilon},$$ which contradicts
\[\|W_0\|_{{\LL^2({S}_{-\varepsilon},T)}}
+\|w_0\|_{[\WW^{1,2}(S_{-\varepsilon})]'}
+\|v_0\|_{[\WW^{1,2}({S}_{-\varepsilon})]'}
+\|V_0\|_{[\WW^{2,2}({S}_{-\varepsilon},T)]'}
=1.
\]

\end{proof}

{\bf Proof of Theorem \ref{t1.1}.}\,\,\,Fix $\varphi\in C^1(S)$ such that $\supp \varphi\subset\subset S_{-\varepsilon}$ and $\varphi|_{S_{-\varepsilon/2}}=1$. For any $y=W+w\mathbf{n}\in C^2(S,\R^3)$, Theorem \ref{t3.4} yields 
\beq
&&\|\varphi W\|_{\LL^2({S}_{-\varepsilon},T)}
+\|\varphi w\|_{[\WW^{1,2}({S}_{-\varepsilon})]'}
\nonumber\\
\leq &&C\Big(\|\sym D(\varphi W)+\varphi w\Pi\|_{\LL^2({S}_{-\varepsilon},T^2)}
+\|W\|_{\LL^2(\Ga_{b_1},T)}
+\|w\|_{\WW^{-1,2}(\Ga_{b_1})}
\Big)\nonumber\\
\leq &&C\Big(\|U\|_{\LL^2(S,T^2)}
+\|W\|_{\LL^2({S}_{-\varepsilon}\backslash{S}_{-\varepsilon/2},T)}
+\|W\|_{\LL^2(\Ga_{b_1},T)}
+\|w\|_{\WW^{-1,2}(\Ga_{b_1})}\Big),
\label{varphiy}
\eeq
where $U=\sym DW+w\Pi$. Define the subdomain
$$S_{-\varepsilon/2}^c
=S\backslash S_{-\varepsilon/2}
=\{\,\a(t,s)\in S\,|\,(t,s)\in(0,a)\times(-b_0,-\varepsilon/2]\,\}.$$ By (\ref{Pi2.801}), $S_{-\varepsilon/2}^c$ is a  non-characteristic region. By \cite[Theorem 1.1]{Yao2019} there exists a constant $C>0,$ independent of $y,$ such that
\be\|W\|_{\LL^2(S_{-\varepsilon/2}^c,T)}\leq C(\|U\|_{\LL^2(S_{-\varepsilon/2}^c,T^2)}+\|W\|_{\LL^2(\Ga_{-b_0},T)}).\label{3.52}\ee 
From this, we derive
\beq
\|w\|_{[\WW^{1,2}(S_{-\varepsilon/2}^c)]'}
&&\leq C(\|U\|_{[\WW^{1,2}(S_{-\varepsilon/2}^c,T^2)]'}+\|\sym DW\|_{[\WW^{1,2}(S_{-\varepsilon/2}^c,T^2)]'})\nonumber\\
&&\leq  C(\|U\|_{\LL^2(S_{-\varepsilon/2}^c,T^2)}+\|W\|_{\LL^2(S_{-\varepsilon/2}^c,T)}
+\|W\|_{\LL^2(\Ga_{-b_0}\cup\Ga_{-\varepsilon/2},T)})\nonumber\\
&&\leq C(\|U\|_{\LL^2(S,T^2)}
+\|W\|_{\LL^2(\Ga_{-b_0}\cup\Ga_{-\varepsilon/2},T)}
),\label{3.53}\eeq 
where the assumption
$\Pi(\a_t,\a_t)>0\qflq x\in\overline{S_{-\varepsilon/2}^c}$ is invoked. Moreover, by \cite[Lemma 3.6]{Yao2019},
\be\|W\|_{\LL^2(\Ga_{-\varepsilon/2},T)}\leq C(\|U\|_{\LL^2(S_{-\varepsilon/2}^c,T^2)}+\|W\|_{\LL^2(\Ga_{-b_0},T)}).\label{3.54}\ee
Combining \eqref{varphiy}--\eqref{3.54} yields the desired inequality \eqref{1.3} by density. \hfill$\Box$\\

{\bf Ethical approval}

This article does not contain any studies involving human participants or animals conducted by the author.

{\bf Declaration of competing interests}

The author declares that there is no known conflict of interest.

 \end{document}